\documentclass{article}

\RequirePackage{amsthm}
\usepackage{lmodern}
\usepackage{graphicx, geometry}
\usepackage{graphicx}%
\usepackage{multirow}%
\usepackage{amsmath,amssymb,amsfonts}%
\usepackage{amsthm}%
\usepackage{mathrsfs}%
\usepackage[title]{appendix}%
\usepackage{xcolor}%
\usepackage{textcomp}%
\usepackage{manyfoot}%
\usepackage{booktabs}% 
\usepackage{comment}%

\usepackage{soul,xcolor}

 \newcommand{\R}{\mathbb{R}}

 \newcommand{\dpiu}{D^{+}}
 \newcommand{\dmen}{D^{-}}

\definecolor{mygreen}{rgb}{0.5, 0.68, 0.178}
 
\newcommand{\KEYWORDS}[1]{
	\noindent\textbf{Keywords:} #1
}
\newcommand{\AMSSUBJ}[1]{
	\noindent\textbf{AMS Subject Classification:} #1
} 

\theoremstyle{thmstyleone}%
\newtheorem{theorem}{Theorem}%  meant for continuous numbers
\newtheorem{proposition}[theorem]{Proposition}% 
\newtheorem{lemma}[theorem]{Lemma}
\newtheorem{corollary}[theorem]{Corollary}

\theoremstyle{thmstyletwo}%
\newtheorem{remark}[theorem]{Remark}%

\theoremstyle{thmstylethree}%
\newtheorem{definition}[theorem]{Definition}%
 
\newtheorem{assumptions}[theorem]{Assumptions} 
\newtheorem{notations}[theorem]{Notations} 
\begin{document}

\title{Discrete  reaction-diffusion system with stochastic dynamical boundary conditions: convergence results }
 
\author{Francesca Arceci\thanks{Department of Mathematics "F. Enriques", University of Milano, Via C. Saldini 50, 20133 Milano, Italy, email:\emph{Francesca.Arceci@unimi.it}}, Francesco Carlo De Vecchi\thanks{Department of Mathematics "F. Casorati", University of Pavia, Via Ferrata, 5, 27100 Pavia, Italy, email: \emph{Francescocarlo.Devecchi@unipv.it}}, Daniela Morale\thanks{Department of Mathematics "F. Enriques", University of Milano, Via C. Saldini 50, 20133 Milano, Italy, email:\emph{Daniela.Morale@unimi.it}}\\
and Stefania Ugolini\thanks{Department of Mathematics "F. Enriques", University of Milano, Via C. Saldini 50, 20133 Milano, Italy, email:\emph{Stefania.Ugolini@unimi.it}}}
\date{}

%\author[1]{\fnm{Francesca} \sur{Arceci}}\email{Francesca.Arceci@unimi.it}
%\author*[2]{\fnm{Francesco} \sur{De Vecchi}}\email{francescocarlo.devecchi@unipv.it}

%\author[1]{\fnm{Daniela} \sur{Morale}}\email{Daniela.Morale@unimi.it}
%\author[1]{\fnm{Stefania} \sur{Ugolini}}\email{Stefania.Ugolini@unimi.it}

%\affil[1]{\orgdiv{Department of Mathematics "F. Enriques"}, \orgname{University of Milano}, \orgaddress{\street{Via C. Saldini, 50}, \city{Milano}, \postcode{20133},   \country{Italy}}}

%\affil[2]{\orgdiv{Department of Mathematics "F. Casorati"}, \orgname{University of Pavia}, \orgaddress{\street{Via Ferrata, 5}, \city{Pavia}, \postcode{27100},   \country{Italy}}}

\maketitle
\begin{abstract}A space discrete approximation to a highly nonlinear reaction-diffusion system  endowed with a stochastic dynamical boundary condition is analyzed and the convergence of the  discrete scheme to the solution to  the corresponding continuum random system   is established. A splitting strategy allows us to decompose the random system into a space-discrete heat equation with a stochastic boundary condition, and a nonlinear and nonlocal space-discrete differential system coupled with the first one and with deterministic initial and boundary conditions.  The convergence result is obtained by first establishing  some a priori estimates for both  space-discrete splitted variables and then exploiting  compact embedding theorems for time-space Besov spaces on the positive lattice. The convergence of a fully discrete approximation of the random system is also discussed.
\end{abstract}
\\
\AMSSUBJ{60H35, 65M06, 65M12}

\section{Introduction}

In the present paper we study a semi-discrete approximation in $[0,T]\times h\mathbb Z_+$, $T\in \mathbb R_+$ and $h\in \mathbb R_+$, of the following random reaction-advection-parabolic equation strongly coupled with an ordinary differential equation
\begin{equation} \label{eq:deterministic_intro}
	\begin{split}
	\frac{\partial }{\partial t}  (\varphi s)   &=	\nabla  \cdot(\varphi  \nabla s )  -\lambda \varphi s c ,   \quad\; \text{in } (0,T]\times (0,+\infty);\\
	\partial_t c &= -\lambda \varphi s c,   \qquad \qquad\qquad \text{in } (0,T]\times (0,+\infty),
	\end{split}
\end{equation}
with $\lambda\in \mathbb R_+$. 
\begin{equation}
\varphi=\varphi(c)=A+Bc
\end{equation}
According with the literature on the topic, we assume (see, e.g., \cite{2004_ADN},\cite{2025_SPA_MauMorUgo}) that the function $\varphi$ depends linearly upon $c$, i.e.   
\begin{equation}\label{eq:linear_porosity}
	\varphi(c) = A+Bc,
\end{equation}
 with $A>0$, $B\neq 0$. By the scaling $s^\prime =|B|s$, $c'=|B|c/A$, $\varphi^\prime(c^\prime)=A\varphi(c)$, we can restrict ourselves to the case
\begin{equation}
    A=1,\quad B=\pm 1;\label{eq:AB_hp}.
\end{equation}
 We observe that in the particular case of marble sulphation the value $B=-1$ is more appropriate \cite{2025_SPA_MauMorUgo}. In any case, in the present paper, we always assume \eqref{eq:linear_porosity}, \eqref{eq:AB_hp} and that
\begin{equation}
\label{eq:limit_porosity}
 0<\varphi_{min}\le \varphi \le \varphi_{max}<1.\end{equation}
The PDE-ODE system \eqref{eq:deterministic_intro} is coupled with bounded  positive initial conditions $   s(x,0)= s_0(x), \,
c(x,0)= c_0(x)$, for any $x \in (0,\infty)$,   endowed with a   random dynamical boundary condition  for $s$  
\begin{equation}\label{eq:stochastic_boundary_condition}
s(t,0)=\psi_t,  \quad t\in [0,T].
\end{equation}
The process  $\psi=\{\psi_t\}_{t\in[0,T]}$ is defined and adapted to a filtered probability space $(\Omega,\mathcal{F},\{\mathcal{F}_t\}_{t\in [0,T]},P)$ and  we suppose that it is a H\"older-continuous bounded noise. More precisely, given  $\eta\in \mathbb R_+$, it is such that, for any  $   t \in [0,T]$,
\begin{equation}\label{eq:regularity_boundary_process} 
 \psi_t\in [0,\eta], \quad     \psi_t \in C^\beta([0,T]), \quad \forall \,\beta\in (1/4,1/2).
\end{equation} 
 A discrete version of the process $(s,c)$ on the positive space lattice $h\mathbb Z_+$ is derived by taking into account both forward and backward finite difference derivatives.
 We adopt a suitable splitting technique that allows  us to face the irregularity of the dynamical boundary condition, by considering $s=u+v$, as in \cite{2024AMU_numerico,2025_SPA_MauMorUgo}. The study is then divided into two different subproblems in $[0,T]\times h\mathbb Z_+$. The first is a space-discrete heat equation for $u$ with zero initial condition and coupled with the stochastic process $\psi$ at the boundary, so that we can take advantage of the regularization effect of the discrete heat equation to compensate for the irregularity of the stochastic boundary path. The second problem is a space-discrete  nonlinear and nonlocal  system for  $v$, whose evolution also depends upon the given $u$ and its derivative, but endowed with a deterministic initial condition and a zero boundary condition. \\ 

The deterministic version of  system \eqref{eq:deterministic_intro} with deterministic, often constant, boundary conditions, has been introduced in \cite{2007_AFNT_TPM,2004_ADN}. It models the degradation of calcium carbonate stone due to pollutant agents, in particular the sulphur dioxide, a very important issue in  Cultural Heritage. 
In system \eqref{eq:deterministic_intro} the function  $s$ stands for the porous concentration of sulphur dioxide that reacts with the calcite $c$  with a rate $\lambda\in \mathbb R_+$ and diffuses in a material with  porosity $\varphi=\varphi(c)$ given by \eqref{eq:linear_porosity}.  As mentioned, porosity,  the fraction of void volume, is governed by a linear dependence on calcite with coefficients $A$ and $B$. In particular, if we consider the marble sulphation phenomena, then the coefficient $B$ should be negative since the calcium carbonate at the beginning of the reaction is transformed into gypsum which has a lower porosity. \\

Model \eqref{eq:deterministic_intro} has been the starting point of a widespread literature, both analytical and numerical,  \cite{2007_AFNT_TPM,2019_BCFGN_CPAA,2023_Bonetti_natalini_NLA,2005_GN_NLA}, where it has been studied on the half-line $\mathbb R_+$ and the boundary conditions are either constant or a given pre-selected bounded measurable positive function   \cite{2004_ADN,2005_GN_NLA}. 
While the mathematical modelling of marble sulphation has primarily been studied from a continuum macroscopic perspective, in the last couple of years to capture the intrinsic randomness of pollutant distribution and environmental fluctuations, recent studies have introduced stochastic elements into the modelling framework. 
At the microscopic level, stochastic interacting particle systems have been proposed in \cite{2025_JMMRU_Arxiv,2024_MACH2023_particles}, while in \cite{2024_DMLTSU_Arxiv} a fully probabilistic interpretation of the deterministic system \eqref{eq:deterministic_intro}   has been provided, relying on an interacting particle system of McKean-Vlasov type.\\

In this work, we study the model proposed in \cite{2023_Arceci_Giordano_Maurelli_Morale_Ugolini,2024_MACH2023_PDE,2025_SPA_MauMorUgo}, in which random noise has been introduced at the macroscale: with the aim to capture the inherent variability of the pollutants in the atmosphere, a random dynamical boundary condition has been coupled to system \eqref{eq:deterministic_intro}. Indeed, a Pearson process at the left boundary $x=0$ is considered, a solution to a Brownian motion-driven stochastic differential equation (SDE) with a mean reverting drift and a squared diffusion coefficient, which is a second-order polynomial of the state. More precisely, the process $\psi$ in \eqref{eq:stochastic_boundary_condition} is chosen as  the unique bounded solution of the SDE:
\begin{equation}\label{eq:Pearson_SDE_boundary}
	d\psi_t = \nu_1(\gamma-\psi_t) dt +\nu_2 \sqrt{\psi_t\left(\eta-\psi_t\right)}dW_t,
\end{equation}
with $\nu_1,\nu_2,\gamma,\eta\in \mathbb R_+$, and $\gamma\le \eta$ and $W\equiv\{W_t\}_{t\in [0,T]}$  a standard Wiener process. The solution of \eqref{eq:Pearson_SDE_boundary} is an example of the random boundary process which is included in the present study. Indeed, in \cite{2025_SPA_MauMorUgo} has been stressed how the regularity stated in \eqref{eq:regularity_boundary_process} is satisfied by the process $\psi$ in \eqref{eq:Pearson_SDE_boundary}. The authors also establish the existence and uniqueness of a mild solution of the system \eqref{eq:deterministic_intro} with constant initial condition and with a boundary condition given by \eqref{eq:Pearson_SDE_boundary}.  A splitting strategy, which we propose again in this paper, has been employed to overcome the difficulties arising from the low regularity at the boundary. More precisely, results in \cite{2025_SPA_MauMorUgo} are achieved under the hypothesis that   $\psi$  in \eqref{eq:stochastic_boundary_condition} has the low regularity \eqref{eq:regularity_boundary_process}. This is also the case in the present paper. Indeed, the results we discuss here hold for the general class of processes with $\beta$- H\"older path continuity with $\beta\in (1/4,1/2)$.
This process class is quite large, including, for example, the fractional Brownian motion with Hurst index $H>{1}/{4}$, or the solution to rough path equations with driving noise as in \eqref{eq:regularity_boundary_process}.
A numerical analysis of the system \eqref{eq:deterministic_intro} coupled with \eqref{eq:Pearson_SDE_boundary}
is provided in \cite{2024AMU_numerico}: a fully discrete forward time centered space approximation is considered coupled with a stochastic path as  boundary condition. The latter is obtained via a Lamperti transformation, to overcome the lack of monotonicity and the problem that the Euler-Maruyama approach does not preserve the bounded domain   \cite{2021_chen_Gan_wang,2023_chen_LSST,1992_Kloeden_Platen,2012_kloeden}. Boundedness and $L^2$ stability results are discussed. With the present work we aim to face the convergence challenge.\\

From an analytical point of view, evolution problems involving stochastic dynamical boundary conditions have been examined particularly focusing on scenarios where the solution exhibits white noise behaviour at the boundary \cite{2015_barbu_bonaccorsi_Tubaro_JMAA,2006_Bon_Zig,2015_Brzezniak_russo,2004_chueshov,1993_daPrato_Zabczyk,2007_debussche_Fuhrman_Tessitore}. Only in \cite{2006_Bon_Zig,2004_chueshov}   PDE boundary conditions are governed by a SDE.  A process driven by Brownian motion appears to be more suitable for Dirichlet boundary conditions, as we consider here. From a mathematical perspective, the regularity of the noise we assume is better than that of the white noise. On the other hand, such an irregular boundary condition is incompatible with the rest of the equation, since it is strongly nonlinear and contains nonlinear terms involving the derivatives of the solution. This is different from the standard literature on SPDEs with random boundary conditions, where the equation contains only local nonlinearities. Furthermore, the situation considered here is significantly more irregular than the boundary conditions in the existing  sulphation deterministic literature \cite{2019_BCFGN_CPAA,2005_GN_NLA,2007_GN_CPDE,2023_Bonetti_natalini_NLA}.\\

As the numerical studies regard, to the best of our knowledge, only a few works deal with the numerical approximation by adopting finite differences \cite{jovanovic_finite_diff,lions1969methodes} of nonlinear parabolic equations in which the boundary conditions has low regularity and possibly randomness as in the present case. Since in the literature the boundary conditions do not evolve in time, from this perspective, one novelty of the present paper is that we consider a highly nonlinear PDE endowed with a stochastic boundary condition of dynamical type and the second one is that the regularity conditions assumed for such a process are shared by a not small class of stochastic processes. 
We stress that one of the main difficulties in handling the numerical convergence of nonlinear random equations having a highly irregular solution (with respect to the standard literature on the topic) is to individualize a suitable space of convergence for the solutions. We adopt the time-space Besov spaces, which are implicitly considered in the linear framework through the study of convergence of numerical equations in fractional Sobolev spaces (see, e.g., \cite{jovanovic_finite_diff,lions1969methodes}), since these spaces allow us to track more precisely the (fractional) regularity of the solution and to exploit the useful interpolation results. In order to use these instruments and handle the non-linearity of the equation, we take inspiration from the study of singular SPDEs on the lattice, see, e.g., \cite{devecchi_nicolay_2021elliptic,ErhardHairer2019,HairerMatetski2018,HairerSteele2002,gubinelli2018pde,martin2019paracontrolled}. We hope that the present paper can contribute to improving the knowledge of these lattice approximation methods for singular SPDEs, that these techniques can be spread more widely and used in the study of the convergence of numerical methods to irregular solutions to differential equations.  \\

The final goal of the present article is to prove the convergence of the numerical scheme associated with \eqref{eq:deterministic_intro} to the unique mild solution of the random PDE-ODE system \eqref{eq:deterministic_intro}, endowed with the stochastic dynamic boundary condition \eqref{eq:Pearson_SDE_boundary}, whose existence and uniqueness have been established in \cite{2025_SPA_MauMorUgo}.
To be more precise, we first study in detail the random space-discrete heat equation, by deriving the fundamental solution on the positive lattice on $[0,T]\times  h\mathbb{Z}_+$, as an odd reflection of the fundamental solution in the whole $h\mathbb{Z}$ and then the explicit representation of the solution of the initial-boundary problem in terms of the difference derivative of the discrete heat kernel on $h\mathbb{Z}$. We provide some a priori estimates for the random space-discrete heat equation solution in the time-space Besov spaces. In particular, we use a recent result on the discrete heat kernel obtained in \cite{devecchi_nicolay_2021elliptic}, which states that the heat-kernel regularizes a function on the lattice-based Besov spaces $B^s_{p,q}$ by gaining some regularity in space.  As a final achievement, we prove that the $L^\infty_tW^{p,r}_x -$ norm of $u$ is controlled by the $C^\beta-$ norm of $\psi$ with $\beta \in (1/4,1/2)$  (Proposition \ref{proposition:estimate_u}). 
Concerning the process $(v,c)$,   the equation satisfied by $v$ is nonlinear and nonlocal; furthermore, it depends on the given $u$ and its derivative $\partial_x u$. To face this second system, we first introduce a simplification strategy, given by a sort of linearization of the whole system. Following \cite{2007_GN_CPDE,2025_SPA_MauMorUgo}, we study a quasi-linear system in which the ODE depends on a function $f$ with good regularity properties, instead of $s$, so it becomes an independent equation.   
We also prove a useful a priori bound for the $L^2-$ norm of $v$ and its discrete derivative (Proposition \ref{proposition:estimates_v}). Although we  mainly follow the line as  in \cite{2025_SPA_MauMorUgo}, we need to specialize the estimates in the adopted discrete spaces, by taking a different route to estimate the term $v_t^2$ and by considering an interpolation between the space $L^2_tH^1_x$ and $L^\infty_tL^2_x.$  
Given the estimates for the splitting variables $u$ and $v$, we finally obtain an a priori estimate for the whole solution $s$ of the quasi-linear system (Proposition \ref{estimate_for_s}), as well as for the original highly non-linear system (Proposition \ref{remark:fs}).\\

The main convergence result is obtained via piecewise interpolation and compactness properties of discrete time-space Besov spaces. More precisely, we first introduce a simple \emph{extension operator} given by a piecewise constant function associating to any function defined on the lattice $h\mathbb{Z}_+$ a function on the continuum $\R_+$. This operator has good continuity and convergence properties with respect to both the discrete and the continuous Besov spaces used in the present paper, see \cite{devecchi_nicolay_2021elliptic}. Furthermore, we also introduce \emph{discretization operator}, which associates with each regular function on $\R_+$ a value at each lattice point given by the mean value of the function on a square centered on the given lattice point. By considering a weak formulation of our space-discrete system, we first state that a weak solution is also a mild solution (Lemma \ref{weak_versus_mild}). Finally, we prove the convergence of the piecewise constant extension of a solution of the space-discrete random system to the unique solution of the random system \eqref{eq:deterministic_intro} in $L^p([0,T],B^k_{p,p}(\R_+))$, for any $p\in [1,2]$ and for any $k<1/p$. The convergence of a fully discrete scheme to the random continuous system \eqref{eq:deterministic_intro}-\eqref{eq:stochastic_boundary_condition} concludes the analysis.\\

The paper is organized as follows. In Section 1 the well-posedness of the random PDE-ODE system is recalled and its numerical approximation is briefly summarized. The space-discrete setting, including a short introduction to discrete time-space Besov spaces, is presented in Section 3. In Section 4 the space-discrete version of the original random PDE-ODE system is derived and the slitting strategy is proposed. The space-discrete heat equation endowed with a zero initial condition and a random boundary condition is studied in details in Section 5. Some a priori estimates for its solution are established, including the Dirac-delta semigroup regularization via Besov spaces properties and the continuity properties of the solution with respect to the boundary process. A priori estimates for a quasi-linear version of the discrete system are provided in Section 6 as well as the boundness of  both the splitting variables via the Feynman-Kac representation formula. Very useful uniform $L^2$ estimates for both the splitting variables which allow to prove analogous estimates for the original variables are also presented.  Section 7 contains all main convergence results. The convergence of the weak solution of the random space-discrete system to the weak solution of the original random PDE-ODE one is established. Considering a fully discrete scheme, its convergence to the original system in the continuum is also investigated. Finally, the proof of the Feynman-Kac formula for a time continuous process defined on a lattice is sketched in Appendix A.

\section{Random PDE - ODE : well-posedness}

The motivation of the present  work is the recent study of the random PDE-ODE system \eqref{eq:deterministic_intro}-\eqref{eq:stochastic_boundary_condition} in the case  where the stochastic process $\psi$ is a specific example, given by the Pearson process solution of the SDE \eqref{eq:Pearson_SDE_boundary}. In this section, we briefly present the main results about the properties of the process $\psi$ and of the well-posedness of the complete random system, as well as of the specific numerical study \cite{2024AMU_numerico,2025_SPA_MauMorUgo}.     
First of all, we stress that the solution $\psi$ to the SDE in \eqref{eq:Pearson_SDE_boundary} satisfies the regularity properties \eqref{eq:regularity_boundary_process}.

\begin{proposition}[A random bounded boundary process \cite{2025_SPA_MauMorUgo}] \label{prop:wellposedness_psi} Let  $\nu_1,\nu_2,\gamma,\eta\in \mathbb R_+$ with $\gamma\le \eta$, and  $\psi_0 \in [0,\eta]$.  The equation \eqref{eq:Pearson_SDE_boundary} satisfies the following properties. 
\begin{enumerate}
    \item[a.] The equation \eqref{eq:Pearson_SDE_boundary} admits a solution $\psi_t, t\in [0,T]$, which is pathwise unique for any initial condition  $\psi_0\in [0,\eta]$. 
    \item[b.] If the parameters are such that $ \,\, 2\nu_1 \gamma \wedge 2\nu_1(\eta-\gamma) \ge \nu_2^2 \eta,$  then  the process $\psi$ stays bounded, i.e. for any $t\in (0,T],  
	 \psi_t \in (0,\eta)$ and, in particular, $\psi_t\in L^\infty(\mathbb R_+)$. 
 \item[c.] The solution paths are H\"older continuous, i.e. for any $\omega\in \Omega$, the  trajectory $\psi\cdot(\omega) \in C^\beta([0,T])$ with $ \beta \in (0,1/2)$.
 \end{enumerate}
\end{proposition}
In order to give a proper definition of the solution $(s,c)$ of  the  system  \eqref{eq:deterministic_intro} with the linear model for the porosity  \eqref{eq:linear_porosity}, it is convenient to consider the following form, that makes evident the high non-linearity of $s$  
\begin{eqnarray} 
    \partial_t    s   &=&	\partial_x^2 s + \tilde{b}_c(t,x) \partial _x s + \gamma_c(t,x) s(Bs- 1),  \label{eq:s2}\\
    \partial_t c &=& - \lambda s (A+B c) c.  \label{eq:c2}
\end{eqnarray}	
Functions $\tilde{b}, \gamma_c $ show the dependence of $s$ on the field $c$ and  are defined as follows \begin{equation} \label{eq:b_c_gamma_c}
    \tilde{b}_c(t,x)=\frac{B}{(A + Bc)} \partial_x c, \qquad 
    \gamma_c(t,x) = \lambda c.  
\end{equation} 
In \cite{2025_SPA_MauMorUgo} the well-posedness of the system \eqref{eq:s2}-\eqref{eq:c2} with the boundary equation \eqref{eq:stochastic_boundary_condition} is proven with respect to the following definition of mild solution to the system.  
\begin{definition}[Bounded positive mild solution]\label{def:mild_solution_s_nonlin}
	A bivariate process $(s,c)$, with $s\in L^\infty([0,T],W^{1,2}(\mathbb R_+))\cap L^\infty\left([0,T]\times \mathbb R_+\right)$ and $c\in B_{b}\left([0,T]\times \mathbb R_+\right)$, the space of bounded Borel functions, is \emph{a bounded positive mild solution} of the nonlinear PDE-ODE system \eqref{eq:s2}-\eqref{eq:c2}  if the following conditions hold.
    \begin{itemize}
        \item[1.] For any $x \in \mathbb R_+$, the function $c(\cdot,x)$ solves the ODE \eqref{eq:c2}, i.e. for any initial bounded function $c_0$ and any $t\in [0,T]$ it is  explicitly given by 
	\begin{align}\label{eq:df_mild_solution_c}
	c(t,x) = \frac{Ac_0(x)}{\varphi(c_0(x))e^{\lambda \int_0^t As(\tau,x)d\tau} -Bc_0(x)}.
	\end{align}
  
        \item[2.] For  any $(t,x)\in [0,T]\times \mathbb R_+,$  the function $s$ is such that  $ s(t,x) \in (0,\eta).$ 
	
        \item[3.] The function $s\in L^\infty([0,T],W^{1,2}(\mathbb R_+))$ is a mild solution of \eqref{eq:s2}, that is
	\begin{equation*} 
	\begin{split}
		s(t,\cdot) &= -2\int^t_0\partial_xG(t-\tau,\cdot) \psi(\tau)d\tau +  {G}(t,\cdot)*_D s_0\\
		&\quad +\int_0^t  {G}(t-\tau,\cdot)*_D\left(\tilde{b}_c(\tau,\cdot)\partial_x s(\tau,\cdot)+\gamma_c(\tau,\cdot) s(\tau,\cdot)\right)(Bs(\tau,\cdot)-1)) ds,
	\end{split}
    \end{equation*}
	 where   $\tilde{b}_c$ and $\gamma_c$ are defined in   \eqref{eq:b_c_gamma_c}, $G$ is  the heat kernel     
and $$  G(t,\cdot)*_Df(x)=\int_0^{+\infty} (G(t,x-y)-G(t,x+y))f(y)dy.$$ 
    \end{itemize}
\end{definition} 
The global existence of a pathwise unique mild solution for the nonlinear equation \eqref{eq:s2}-\eqref{eq:c2}  is proved in \cite{2025_SPA_MauMorUgo} under some regularity hypotheses on the initial condition and by tightening the range of the H\"older coefficient of the boundary condition, as follows.
\begin{theorem}[Well-posed continuum random system \cite{2025_SPA_MauMorUgo}]\label{teo:wellposedness_continuuum_system}
Let us consider the system \eqref{eq:s2}-\eqref{eq:c2}  on $[0,T]\times \mathbb R_+$. Suppose that the hypotheses of  Proposition \ref{prop:wellposedness_psi}  are satisfied with $\eta <1$ if $B=1$ and that the solution $\psi$ of  equation \eqref{eq:Pearson_SDE_boundary} is such that
$$
    \psi\in C^\beta,\, \beta \in (1/4,1/2),\quad    \psi(0)=0, \quad    0\le \psi_t\le \eta.
$$
Furthermore, suppose that the initial conditions have the following regularity
$$
    0\leq s_0(x) \leq \eta, \quad s_0 \in W_x^{1,2}, s_0(0)=0, \qquad  0<m\leq c_0(x) \leq C_0, \quad
     C_0- c_0\in  W_x^{1,2}.
$$
Then,    system \eqref{eq:s2}-\eqref{eq:c2} admits   a pathwise unique mild solution $(s,c)$, according with Definition \ref{def:mild_solution_s_nonlin}.
\end{theorem}
 The proof of Theorem \ref{teo:wellposedness_continuuum_system} employs a splitting strategy, where the solution 
$(s,c)$ to the system \eqref{eq:s2}-\eqref{eq:c2} is expressed as 
$(u+v,c)$, where $u$ represents the solution to the heat equation coupled with the stochastic boundary condition \eqref{eq:Pearson_SDE_boundary} and does not depend on the less regular function $c$. Differently, 
$(v,c)$ is the solution to a nonlinear system that depends functionally on 
$u$, but is characterized by a deterministic constant boundary condition. This approach enables us to offset the irregularity of the stochastic boundary path with the regularity of $u$, arising from the heat equation.

\smallskip

The same splitting strategy is adopted in the numerical discretization of the random model for any realization of the dynamical boundary condition $\psi_t$. The numerical scheme of both the subsystems is obtained via forward in time and centered in space (FTCS) finite difference approximation \cite{2024AMU_numerico}.  The time-space discrete heat equation is coupled with a Lamperti sloping smoothing truncation approximation for boundary dynamics \eqref{eq:Pearson_SDE_boundary}, in order to  face the difficulty of a non-Lipschitz diffusion coefficient and to guarantee  boundedness and  monotonicity \cite{2024AMU_numerico}. 
Given a time-space discretization with fixed time and space meshes,  respectively, the boundedness of the discrete solution and stability of the scheme are established under suitable  conditions.
The interested reader may refer to \cite{2024AMU_numerico} for details.

\section{A space-discrete setting}\label{sec:space-discrete_setting}
To provide a space finite difference approximation of \eqref{eq:deterministic_intro}	 - \eqref{eq:stochastic_boundary_condition}, let us introduce the discrete spaces we will consider in the paper. 
\begin{definition}[Discrete space]\label{def:discrete_space_norms}
  Let $h >0$. Let us define the following  lattice spaces
$$
\Omega_h=h\mathbb{Z}, \quad \Omega_h^+ = h\mathbb{Z}_+, \quad \Omega_{h,0}^+ = \Omega_h^+\cup \{0\}, \quad \Omega_h^- = h\mathbb{Z}_- 
$$ 
that is  if $z\in \Omega_h$, then $z=n h,$ for some $n \in \mathbb{Z}$. 
 For any summable function $f: \Omega_h\to\R$   we often use the following integral notation
$$
    \int_{\Omega_h}f(z)dz = h\sum_{z\in \Omega_h} f(z),
$$
while as usual the space $L^p(\Omega_h)$,   $p\in [1,+\infty]$, is    endowed with the discrete norm \begin{equation}\label{eq:grid_norm_space}
    \| f\|_{L^p(\Omega_h)} : = \left(h \sum_{z\in \Omega_h}|f(z)|^p\right)^{\frac{1}{p}}.
\end{equation} The supremum norm is as usual denoted by $
    \|f\|_{L^\infty(\Omega_h)}$ %\cite{Ladyzenskaya}.
\end{definition}
Let us introduce the discrete Besov space,   a discrete version of the classical   Besov space \cite{Triebel1983}. The interest reader may refer to  \cite{devecchi_nicolay_2021elliptic}.
\begin{definition}[Discrete Fourier transform]
   Let $\mathcal{S}(\Omega_h)$ be the Schwartz test function space defined on $\Omega_h$, meaning the sequence decreasing at infinity less than any polynomial, and by $\mathcal{S}^\prime(\Omega_h)$ its topological dual space, i.e. the set of functions defined on $\Omega_h$ with at most polynomial growth. The Fourier
transform $\mathcal{F}_{h} : \mathcal{S} (\Omega_h)
\rightarrow C^{\infty} \left( \mathbb{T}_{h} \right)$,
where $\mathbb{T}_{h} = \left[ -
 {\pi}/{h},  {\pi}/{h} \right]$, is defined, for any $f\in \mathcal{S} (\Omega_h)$ and $  x \in \mathbb{T}_{h}$ as
\begin{equation} \mathcal{F}_{h} (f) (x) := \hat{f} (x) = \int_{\Omega_h} e^{i z \cdot x} f (z) d z.\label{eq:discreteFourier} \end{equation}
 Then the inverse Fourier transform  $\mathcal{F}^{- 1}_h :
C^{\infty} \left( \mathbb{T}_{h} \right) \rightarrow
\mathcal{S} (\Omega_h)$ is, for any $g\in C^{\infty} \left( \mathbb{T}_{h} \right) $ and $z\in \Omega_h$ 
\begin{equation}\label{eq:inverse_fourier_torous}
 \mathcal{F}^{- 1}_h (g) (z) := \frac{1}{2 \pi}
   \int_{\mathbb{T}_{h}} e^{- i z \cdot x} g (x)
   d x. \end{equation}
\end{definition}
\begin{remark}
    We can extend the Fourier transform
$\mathcal{F}_{h}$, and the inverse Fourier transform
$\mathcal{F}_h^{- 1}$ from the space $\mathcal{S}' (\Omega_h)$ into $\mathcal{D}' \left( \mathbb{T}_{h}
\right)$ (where $\mathcal{D}' \left( \mathbb{T}_{h}
\right)$ is the space of distributions, i.e. the topological dual of
$C^{\infty} \left( \mathbb{T}_{h} \right)$), and from
$\mathcal{D}' \left( \mathbb{T}_{h} \right)$ into
$\mathcal{S}' (\Omega_h),$ respectively.
\end{remark}
\begin{definition}[Partition of the unity and Littlewood-Paley operators] Let $\{\varphi_j\}_{j\geq -1}$ be a dyadic partition of unity of $\mathbb{R}$ \cite{bahouri2011fourier}. %(\cite[Section 2.2]{bahouri2011fourier}). 
We introduce a suitable partition of unity $(\varphi_j^h)_{-1\le j\le J_h}$ for the one-dimensional torus $\mathbb{T}_h$ as follows:   for any   $x\in \mathbb{T}_h$ 
\begin{equation*}
    \varphi_j^h(x) = \begin{cases}
        \varphi_j(x) & \text{ if } j<J_h\\
        1-\sum_{j=-1}^{J_h-1}\varphi_j(x) & \text{ if } j= J_h
    \end{cases}, 
\end{equation*}
where
\begin{equation*}
    J_h = \min\left\{j\ge -1,\, supp(\varphi_j)\not\subset\mathbb{T}_h \right\}.
\end{equation*}
The associated discrete Littlewood-Paley block operators are  the operators defined for any $-1 \leq j \leq J_h$  as $\Delta_j: \mathcal{S}'(\Omega_h) \rightarrow \mathcal{S}'(\Omega_h)$ such that
\[\Delta_j f=\mathcal{F}^{-1}_h(\varphi^h_j\cdot \mathcal{F}_h(f)).\] 
\end{definition}
 \begin{definition}[Discrete Besov spaces \cite{devecchi_nicolay_2021elliptic}]
    For any $p,q\in[1,+\infty],\, \alpha\in \R$, we define  the Besov space $B_{p,q}^\alpha(\Omega_h)$ as the  Banach space, subset of $\mathcal{S}^\prime (\Omega_h),$ such that  $f\in B_{p,q}^\alpha(\Omega_h)$ if  
     \begin{equation}\label{eq:besov_norm}       \left\|f\right\|_{B_{p,q}^\alpha(\Omega_h)} = \left(\sum_{j=-1}^{J_k} \left(2^{j\alpha}\|\Delta_j f\|_{L^p(\Omega_h)}\right)^q\right)^{\frac{1}{q}} < \infty.
    \end{equation}The parameter $\alpha, p$  and $q$ describe the smoothness of the function,   its integrability and   an additional refinement of the regularity scale, respectively.
\end{definition}
\begin{remark} 
   If $p=q=2$, then $B^\alpha_{2, 2}=H^\alpha$, i.e. the traditional fractional Sobolev space, and, whenever $0<\alpha<1$, $B^\alpha_{\infty, \infty}=C^\alpha,$ i.e. the space of $\alpha$-H\"older regular functions. 
\end{remark}

 \begin{proposition}
    \label{prop:Delta_in_Besov}
     Let $k>1$, $p,q\in[ 1,+\infty]$ and  $ \alpha < 1/p-1$. Then,  $\delta_0 \in B^{\alpha}_{p,q}(\Omega_h^+)$. 
\end{proposition} \begin{proof}
Since $\mathcal{F}_h(\delta_0)=1$ and, by definition, $\mathcal{F}^{-1}_h(\varphi^h_j  ) $ is about $2^j$ around zero and is supported  on a length $2^{-j}$, so that  the number of discrete points is about $2^{-j}/h$ and 
    \begin{eqnarray*}      \|\Delta_j\delta_0\|^p_{L^p(\Omega_h)} =   \|\mathcal{F}^{-1}_h(\varphi^h_j  )\|^p_{L^p(\Omega_h)} &=&  h \sum_{z\in \Omega_h}\left|\frac{1}{2 \pi}
   \int_{\mathbb{T}_{h}} e^{- i z \cdot x} \varphi^h_j  (x)
   d x.\right|^p\approx  h 2^{j\left(1-\frac{1}{p}\right)}. 
    \end{eqnarray*}
Then, from \eqref{eq:besov_norm}   we get  
\begin{eqnarray*}       \left\|\delta_0\right\|^q_{B_{p,q}^\alpha(\Omega_h)} &=&   \sum_{j=-1}^{J_k} \left(2^{j\alpha}\|\Delta_j\delta_0\|_{L^p(\Omega_h)}\right)^q  \approx  h  \sum_{j=-1}^{J_k} \left(2^{j\left(\alpha+ 1-\frac{1}{p}\right) }\right)^q. 
    \end{eqnarray*}  Then the thesis is achieved when $\alpha<1/p-1.$ \end{proof}

\begin{definition}
 For $p,q,r \in [1,+\infty] $ and $\alpha\in \R $ and for any $t>0$ we consider the space $ L^r\left([0,t],B^\alpha_{p,q}(\Omega_h)\right)$
 with norm given, for any $ f: [0,t] \rightarrow B^\alpha_{p,q}(\Omega_h)$, by
 $$ ||f||_{L^r([0,t],B^\alpha_{p,q}(\Omega_h))}=\left(\int_0^t ||f(\tau)||^r_{B^\alpha_{p,q}(\Omega_h)}d\tau\right)^{1/r}.$$
\end{definition}
\begin{definition}[Continuous time-discrete space Besov space] Let us denote by $\Delta^\tau_j$  the Littlewood-Paley block   with respect to   time.  For $\bar{\alpha}, \alpha\in \R$, let us denote 
$$ B_{r,r}^{\bar{\alpha}} \left([0,t],B^\alpha_{p,q}(\Omega_h) \right)   $$ the space of the functions on $[0,t]\times \Omega_h$ such that 
$$ ||f||_{B_{r,r}^{\bar{\alpha}} ([0,t],B^\alpha_{p,q}(\Omega_h) ) } = \sum_{j\ge -1} \int_0^t \|\Delta^\tau_j f(\tau)\|_{B^\alpha_{p,q}(\Omega_h)} d\tau<\infty.  $$
\end{definition}
The study of immersion properties in the context of the Besov spaces is useful, since we may take advantage of compact embedding properties characterizing such spaces. 
\begin{proposition}[Difference characterization \cite{devecchi_nicolay_2021elliptic,martin2019paracontrolled,2022_Turra}]\label{difference_characterization}
    Let $r,p,q \in [1,+\infty) $ and $\bar{\alpha}, \alpha\in \R$. If $ f\in B_{r,r}^{\bar{\alpha}} \left([0,t],B^\alpha_{p,q}(\Omega_h) \right)$, then
    \begin{equation}\label{eq:difference_characterization_besov}
    \begin{split}
   ||f||_{B_{r,r}^{\bar{\alpha}} ([0,t],B^\alpha_{p,q}(\Omega_h) ) } &\equiv ||f||_{{L^r([0,t],B^\alpha_{p,q}(\Omega_h))} }\\
   &+\int_0^t \int_{|s|<1} \frac{||f(\tau + s,\cdot)-f(\tau,\cdot)||^r_{B^\alpha_{p,q}(\Omega_h) }}{s^{1+r {\bar{\alpha}}}} ds d\tau.
   \end{split}
\end{equation}
\end{proposition}
\begin{proposition}[Compact embedding \cite{devecchi_nicolay_2021elliptic,2022_Turra}]
   Let $\alpha_1,\alpha_2 \in \mathbb R$, $p_i,q_i \in [1,+\infty), i=1,2$ such that the immersion 
    $$B^{\alpha_1}_{p_1,q_1}(\Omega_h)  \subset B^{\alpha_2}_{p_2,q_2}(\Omega_h)  $$
    is compact. Let  $\bar{\alpha},\bar{\alpha}^\prime, r, r^\prime \in \mathbb R$ such that $\bar{\alpha}>\bar{\alpha}^\prime $ and $\bar{\alpha}-\frac{1}{r}>\bar{\alpha}^\prime -\frac{1}{r^\prime}, $  then also the immersion
    $$B_{r,r}^{\bar{\alpha}} \left([0,t],B^{\alpha_1}_{p_1,q_1}(\Omega_h)\right) \subset B_{r^\prime,r^\prime}^{\bar{\alpha}^\prime} \left([0,t],B^{\alpha_2}_{p_2,q_2}(\Omega_h) \right) $$
    is compact.\end{proposition}
For sake of simplicity, when   needed, for the Banach, Sobolev and Besov spaces we introduce the following notations: $L_x^p = L^p(\Omega_h^+),\, W_x^{p,r}=W_x^{p,r}(\Omega_h^+),\, B_{p,q}^\alpha = B_{p,q}^\alpha(\Omega_h^+)$ in space and $L_t L_x^q = L^p([0,T],L^q(\Omega_h^+))$ in time and space.
In the following we restrict to the case $d=1$.

\smallskip

To end this section, let us define the discrete difference operators, counterparts of the differential operators in system \eqref{eq:deterministic_intro}: the forward and backward differences $\dpiu_h, \dmen_h$ and the symmetric Laplacian $\Delta_h$.
  	\begin{definition}[Discrete Laplacian]
 Given a smooth function $f:\Omega_h\rightarrow \R$, let us introduce the shift forward $\tau_h$ and the shift backward $\tau_{-h}$ operators, acting upon a function  $f$, i.e.   $\tau_h (f)  := f(\cdot+h),\, \tau_{-h}(f) :=f(\cdot-h)$. 
 
 Then, the forward and backward divided difference   are defined as
	\begin{equation}\label{eq:finite_derivative}
    \begin{split}
	\dpiu_h f(x) & = \frac{f(x+h)-f(x)}{h}= \frac{(\tau_h -\mathcal{I})f(x)}{h}; \\
	\dmen_h f(x) & = \frac{f(x)-f(x-h)}{h}=\frac{(\mathcal{I}-\tau_{-h})f}{h},
	\end{split}
    \end{equation}
    where $\mathcal{I}f=f.$ 
   We refer to the quantities $\dpiu_h f, \dmen_h f$ in \eqref{eq:finite_derivative} as the \emph{ forward and backward  difference derivatives}.  Finally, we define the (symmetric) \emph{discrete Laplacian }  as
\begin{equation}\label{eq:discreteLaplacian}\Delta_h  f (x)=\frac{1}{2}\left[\dpiu_h \dmen_h+\dmen_h\dpiu_h \right] f(x)=\frac{f(x+h)-2f(x)+f(x-h)}{h^2}.
\end{equation} \end{definition} 
From the definition of the discrete Laplacian \eqref{eq:discreteLaplacian}, we have that $
 \Delta_h^D=\left(\dpiu_h  -\dmen_h\right)/h=\dpiu_h \dmen_h=\dmen_h \dpiu_h$. 

 \begin{remark}[The discrete Laplacian and a uniform homegeneous Markov chain]\label{process_laplacian}

Let us recall the identification of the stochastic process  having the discrete Laplacian as infinitesimal generator.
	Let us introduce the symmetric random walk $\tilde{X}$ as a discrete time Markov chain with transition matrix $ \tilde{P}=(p_{xy})_{xy}$ with $p_{xy}=1/2 $ when $y=x+h $ and $y=x-h$. It is well-know that we can see the transition matrix $\tilde{P}$ as a linear operator on bounded measurable real functions $f$ by:
	\begin{equation}
		\tilde{P} f(x)=\sum_{y}p^{(n)}_{xy}f(x)=E_x[f(X_n)],
	\end{equation}
	and the corresponding infinitesimal generator takes the form
	\begin{equation}
		L=\tilde{P}-I.
	\end{equation}
	Let $N$ be a homegeneous Poisson process on $\R_+ $ with intensity $\lambda>0 $ and assume that $N$ and $\tilde{X} $ are independent. We consider the process $X_t:=\tilde{X}_{N(t)} $, usually called the uniform Markov chain, 
	whose generator assumes the form
	\[
	L=\lambda(\tilde{P}-I).
	\]
	Taking $\lambda=\frac{2}{h^2}, $  the generator of $X_t$ coincides with the discrete Laplacian operator. Indeed
	\[
	\Delta_hf=Lf=\frac{2}{h^2}(\tilde{P}-I)=\frac{2}{h^2}\left[\frac{f(x+h)+f(x-h)}{2}-{f(x)}\right].
	\]
	\end{remark}
  
For convenience of the reader we recall the following (discrete) Leibniz's rule.

\begin{lemma}[Leibniz's rule, \cite{jordan}] \label{lem:Leibniz}
For any $n\ge 0$ the following (discrete) Leibniz's rule  for the discrete forward finite difference of order $n$
$$
    D_h^{+,n} (fg) := \sum_{j=0}^n\binom{n}{j}(D_h^{+,j} f)(D_h^{+,n-j} \tau_h^j g).
$$
\end{lemma}
For the forward and backward finite derivative of order one for the product of two functions, since  for $f,g$ defined in $\Omega_h$,  $fg=gf$ we get
\begin{equation*}
\begin{split}
\dpiu_h (fg) & = f \cdot \dpiu_h g + \dpiu_h f\cdot (\tau_h g) = (\dpiu_h f)g + (\tau_h f)\dpiu_h g;\\
    \dmen_h (fg) &= \dpiu_h(\tau_{-h}(fg)) =
    \tau_{-h}f\cdot\dmen_h g + \dmen_h f\cdot g=\tau_{-h} g \cdot \dmen f + \dmen g \cdot f(x).
\end{split}
\end{equation*}
Then, from   \begin{align*}
    \dpiu_h(f \dmen_h g ) &=  f \dpiu_h(\dmen_h g)  +\dpiu_h f \cdot \tau_h(\dmen_h g) 
     = f\Delta_h g + \dpiu_h f \cdot \dpiu_h g \\    \dmen_h(f\dpiu_h g ) & =  f\dmen_h(\dpiu_h g) + \dmen_h f (\tau_{-h}\dpiu_h g ) =
    f \Delta_h g +\dmen_h f \dmen_h g,
\end{align*}
we get a discrete version of the operator $\nabla\left( f\nabla g\right)$, that is 
\begin{equation}\label{eq:chain_rule_D}
     \left[\dpiu_h(f\dmen_h g) + \dmen_h(f\dpiu_h g)\right]/2 =f \, \Delta_h g  +  \left[\dpiu_h f\, \dpiu_h g+\dmen_h f \, \dmen_h g\right]/2.
\end{equation}
For convenience, we recall the discrete counterpart of the integration by parts formula  and the rule for the forward difference derivative for the ratio of two functions. 
\begin{lemma}[Chain rules]\label{lem:discrete_int_parts}
Let $f,g$ be two functions  in $   \Omega_{h,0}^+  $. The following chain rules hold
\begin{equation*}
\begin{split}
    \int_{\Omega_{h}^+} f(y) \dmen_h g(y) dy &= -f(0)g(0) - \int_{\Omega_{h,0}^+}\dpiu_h f(y) \cdot g(y)dy\\
   \int_{\Omega_{h,0}^+} f(y)\dpiu_h g(y)dy 
    & =  f(0)  g(0) - \int_{\Omega_{h}^+} (\dmen_h f(y)) g(y) dy. 
    \end{split}
\end{equation*} 
and
\begin{equation}\label{eq:chain_rule_D-D+} \begin{split}
     \int_{\Omega_{h}^+} f(y) \dmen_h \dpiu_h g(y) dy &=\dpiu_h f(0) g(0) -f(0)\dpiu_h \, g(0) +  \int_{\Omega_{h}^+}\dpiu_h\dmen_h f(y)\,g(y)dy 
    \end{split}
\end{equation} 
\end{lemma}
\begin{lemma}
    Given two functions $f,g$, then 
    \begin{equation}\label{eq:Leibniz_quotient}
    \dpiu_h \left(\frac{f}{g}\right) = \frac{g(x)\dpiu_h f(x) - f(x)\dpiu_h g(x)}{g(x+h)g(x)}
\end{equation}
\end{lemma}
 
 \section{The space-discrete random  evolution systems}

According with the previous notations, we may write  the semi-discrete version of the  random system \eqref{eq:deterministic_intro} - \eqref{eq:stochastic_boundary_condition}, whose solution is given by the bivariate process $(s,c)$ with $\varphi=\varphi(c)$ in \eqref{eq:linear_porosity} and $\lambda\in \mathbb R_+$, as the following continuum time and  finite difference in space system 
for $(t,x)\in (0,T]\times \Omega_h^+$:
\begin{equation}\label{eq:discretesystem}
\begin{split}
{\partial}_t (\varphi(c) s) & = \frac{1}{2}[\dpiu_h\varphi(c)\dpiu_hs+\dmen_h\varphi(c) \dmen_h s]+\varphi \Delta_h s -\lambda \varphi(c) s c, \\
& =\frac{1}{2}\left[\dpiu_h(\varphi(c)\dmen_h s) + \dmen_h(\varphi(c)\dpiu_h s)\right] -\lambda \varphi(c) s c, \\
{\partial}_t c & = -\lambda  \varphi(c) s c,  \\
\end{split}
\end{equation} 
with  initial and boundary conditions given by
\begin{equation}\label{eq:discretesystem_Initial_boundary}
\begin{aligned}
 & s(0,x)=s_0(x),\quad  c(0,x)=c_0(x), & x\in \Omega_h^+;\\
&s(t,0)=\psi_t, \quad  c(t,0)=Ac_0(x)\left[\varphi(c_0)e^{\lambda A \int_0^t \psi_\tau d\tau} -Bc_0(x)\right]^{-1},\quad &t\in [0,T],\end{aligned}
\end{equation} where the random boundary condition $\psi$   has the regularity \eqref{eq:regularity_boundary_process}.   

\smallskip

Note that from \eqref{eq:discretesystem} with a little algebra we may also deduce an explicit equation for $s$,  by inserting the linear model of $\varphi$ in  \eqref{eq:linear_porosity} and  defining the following functions \begin{equation}\label{eq:functions_bc_gamma}
    \begin{split}
        b_c  = {\frac12} \frac{B}{A+Bc}, \qquad \gamma_c=\lambda c. 
    \end{split}
\end{equation} 
 \begin{proposition}
A discrete in space and continuous in time version of system \eqref{eq:s2}-\eqref{eq:c2} is given by the following system in $(0,T] \times \Omega_h^+ $
\begin{equation}\label{eq:system_s_c_discrete}
\begin{split}
    \partial_t s &= \Delta_h s  + b_c \Big[\dpiu_h c \dpiu_h s+\dmen_h c \dmen_h s\Big] + \gamma_c B s^2- \gamma_c s,   \\
    \partial_t c  &= -\lambda  \varphi s c,    
    \end{split}
\end{equation}
 where the $b_c$ and $\gamma_c$ are given by \eqref{eq:functions_bc_gamma} and the initial-boundary conditions are \eqref{eq:discretesystem_Initial_boundary}.
 \end{proposition}
 
The final goal of the present study is to establish the convergence of the  solution of the semi-discrete  system \eqref{eq:discretesystem}-\eqref{eq:discretesystem_Initial_boundary}
to the unique mild solution  of the system \eqref{eq:deterministic_intro}	 - \eqref{eq:stochastic_boundary_condition}, with the dynamical boundary process satisfying the regularity \eqref{eq:regularity_boundary_process} as in Definition \ref{def:mild_solution_s_nonlin}. 
In order to carry out the study, we follow the same splitting strategy as in   \cite{2025_SPA_MauMorUgo} and we represent the process   $s$   as  the sum of two functions $u+v$, where $u$ is the solution of a space discrete heat equation with zero initial condition and random   boundary condition inherited by the whole system, while $v$ solves a non-linear equation coupled with $u$ and depending upon $c$, seen as a function of $u+v$,  but  with deterministic zero boundary condition and initial condition given by $v_0(x)=s_0(x)-u_0(x)=s_0(x)$.

\begin{proposition}
Let  $u: [0,T]\times \Omega_{h,0}^+\rightarrow \mathbb R_+$ be the solution of the following  random discrete heat equation
\begin{equation*} 
    \begin{cases}
        \partial_t u(t,x)  =\Delta_h u(t,x) = \dmen_h \dpiu_h u(t,x)& (t,x) \in (0,T]\times \Omega_h^+; \\
        u(0,x)  =0, & x\in \Omega_h^+;\\
        u(t,0)  = \psi_t,& t\in [0,T],\\
    \end{cases}
\end{equation*}
with $\psi$ satisfying condition \eqref{eq:regularity_boundary_process}.  
Furthermore, let 
 $v$ be the solution of the following non linear space-discrete system  
    \begin{equation}\label{eq:nonlinear_discrete_system_v}
    \begin{cases}
    \partial_t v = \Delta_h v +b_c(t,x)[\dpiu_hc\dpiu_h v + \dmen_h c\dmen_h v] -\gamma v\\
\quad \quad \quad  +b_c(t,x)[\dpiu_hc\dpiu_h u + \dmen_h c\dmen_h u] \\
\quad \quad \quad-\gamma_c(t,x) u+\gamma_c(t,x)B(u+v)^2, & (t,x) \in (0,T]\times \Omega_h^+;\\
        v(0,x)   = s_0(x), & x\in \Omega_h^+;\\
        v(t,0)   = 0,  & t\in [0,T],
    \end{cases}
    \end{equation}
    with $c$ given by
    \begin{align}\label{eq:c_discrete}
	c(t,x) = A c_0(x)\left[\varphi(c_0(x))e^{\lambda A\int_0^t (u+v)(\tau,x)d\tau} -Bc_0(x)\right]^{-1}, \qquad (t,x) \in [0,T]\times \Omega_h^+.
	\end{align} Putting $s=u+v$, then $(s,c)$ is a solution to the system \eqref{eq:discretesystem} - \eqref{eq:discretesystem_Initial_boundary}.
\end{proposition}
\begin{proof}
The two sub-systems are obtained with a little of algebra, by   $s=u+v $, and exploiting the linearity of the operator $\Delta_h$ and $D^+_h, D^-_h$  as follows
	\begin{align*}
		\partial_t u+\partial_t v & = \Delta_h u+\Delta_h v + b_c(t,x)\Big[\dpiu_h c(\dpiu_h u+\dpiu_h v) +\dmen_h c(\dmen_h u+\dmen_h v)\Big]\\
  &+\gamma_c(t,x)B\left(u^2+2uv +v^2\right)-\gamma_c(t,x) u-\gamma_c(t,x) v.	\end{align*}\end{proof}
  \begin{remark}[Bounded porosity]
      Note that since \eqref{eq:c_discrete} holds, there always exist parameters such that \eqref{eq:limit_porosity} hold true. From now on, we suppose that parameters $A,B$ are chosen in the proper way.
  \end{remark}
\begin{remark}	After our splitting strategy the problem is divided into two different sub-systems: the first is a discrete heat equation for $u$ with zero initial condition and coupled with the stochastic process $\psi$ at the boundary, so that we can take advantage of the regularization effect of the discrete heat equation to compensate for the irregularity of the stochastic boundary path. The second problem is a space-discrete  nonlinear and nonlocal  system for  $v$, whose evolution also depends upon the given $u$ and its derivative, but endowed with a deterministic initial condition and a zero boundary condition. 
\end{remark}

\section{A random space-discrete heat equation}

Let us focus on the study of the semi-discrete heat equation in $[0,T]\times \Omega_{h,0}^+$\begin{equation}\label{eq:discreteheatequation}
    \begin{cases}
        \partial_t u(t,x)  =\Delta_h u(t,x) = \dmen_h \dpiu_h u(t,x)& (t,x) \in (0,T]\times \Omega_h^+; \\
        u(0,x)  =0, & x\in \Omega_h^+;\\
        u(t,0)  = \psi_t,& t\in [0,T],\\
    \end{cases}
\end{equation}
where $\psi$ satisfies  condition \eqref{eq:regularity_boundary_process}. For completeness, we derive an explicit semigroup representation of the solution and then we provide some a priori estimates.

\subsection{The semigroup representation of the solution}

 For the convenience of the reader we discuss the representation of the solution of the problem \eqref{eq:discreteheatequation} into details, starting from the fundamental solution   $H_D(t,x)$.
 
 \begin{proposition}[Fundamental discrete heat solution ]\label{prop:fundamental_heat_solution}
   The solution of the following initial value problem in $\Omega_h$  \begin{equation}\label{eq:discreteheatkernel}
	\begin{cases}
 \partial_t H_D(t,x) &= \Delta_h H_D(t,x), \quad (t,x)\in  [0,T] \times \Omega_h;\\
		H_D(0,x) &= \delta_{0}(x), \qquad \qquad  	x\in  \Omega_h,
	\end{cases}
\end{equation}
where  $\delta_0(x)=\frac{1}{h} 1_{0}(x)$,   is given by \begin{equation}\label{eq:H_D}
	H_D(t,x) = \frac{1}{2\pi} \int_{-\pi/h}^{\pi/h} e^{i\xi x}e^{-\frac{4}{h^2}\sin^2(h\xi/2)t}d\xi.
\end{equation}
 \end{proposition} 
\begin{proof} 
Let $\hat{H}_D$ be the Fourier transform of ${H}_D$ as in  \eqref{eq:discreteFourier}, that is for any $\xi\in \mathbb T_h,$ \[ \hat{H}_D(t,\xi) = h\sum_{x\in \Omega_h} e^{ix\xi} H_D(t,x), \qquad \hat{H}_D(0,\xi) = h \sum_{x\in \Omega_h} e^{ix\xi} \delta_0(x) =1.
\]By the equation in \eqref{eq:discreteheatkernel}, we get  for $(t,\xi)\in  [0,T] \times \mathbb T_h$

    \begin{align*}
    \frac{\partial}{\partial t} \hat{H}_D(t,\xi) &  = h \frac{1}{h^2} \sum_{x\in\Omega_h} e^{ix\xi}H_D(t,x)\left({e^{ih\xi} + e^{-ih\xi} - 2}\right)  = -\frac{4}{h^2} \sin^2\biggl(\frac{h\xi}{2}\biggr) \hat{H}_D(t,\xi),
\end{align*}
 whose solution is, for any $(t,\xi)\in [0,T]\times \mathbb T_h$,
\begin{equation*}
	\hat{H}_D(t,\xi) = \hat{H}_D(0,\xi) e^{-\frac{4}{h^2} \sin^2(h\xi/2)t},
\end{equation*}
By applying the inverse Fourier transform    \eqref{eq:inverse_fourier_torous} we obtain the expression for $H_D$, that is 
\begin{equation*}
\begin{split}
H_D(t,x) &= \frac{1}{2\pi} \int_{-\pi/h}^{\pi/h} e^{-ix\xi} \hat{H}_D(t,\xi)d\xi= \frac{1}{2\pi} \int_{-\pi/h}^{\pi/h} e^{-ix\xi}   e^{-\frac{4}{h^2} \sin^2(h\xi/2)t}d\xi\\ &= \frac{1}{2\pi} \int_{-\pi/h}^{+\pi/h} e^{ix\bar{\xi}}   e^{-\frac{4}{h^2} \sin^2(h\bar{\xi}/2)t}d\bar{\xi}.
\end{split}
\end{equation*}
\end{proof}
\begin{definition}[Discrete heat semigroup] Given the fundamental solution $H_D$, we define the discrete heat semigroup as the linear operator $e^{t\Delta_h}$, such that for any measurable function  $f$
\begin{equation}\label{eq:heat_semigroup_exp}
e^{\Delta_h t} f:=  H_D(t,\cdot)*f (x)=\int_{\Omega_h} H_D(t,x-y)f(y)dy= h\sum_{y\in \Omega_h}  H_D(t,x-y)f(y).
\end{equation}
    
\end{definition}
It follows that the solution $\tilde{u}(t,x)$  to the space-discrete heat equation on the lattice $\Omega_h$ admits the following useful representation in terms of a general initial condition.

\begin{corollary}[Initial value problem]
The solution to the following initial value problem
\begin{equation}\label{eq:heat_equation_Omega_h}
	\begin{cases}
	\displaystyle	\frac{\partial \tilde{u}}{\partial t}(t,x) = \Delta_h \tilde{u}(t,x), & \quad (t,x)\in  [0,T] \times \Omega_h;\\ 
		\tilde{u}(0,x) = g(x) & \quad x\in \Omega_h;\\
	\end{cases}
\end{equation}
admits the following representation in terms of the fundamental solution on $\Omega_h$ 
\begin{equation}\label{eq:HeatSolution_Omega_h}
	\tilde{u}(t,x) = h\sum_{y\in \Omega_h} H_D(t,x-y)g(y)=e^{\Delta_h t}g (x).
\end{equation}
\end{corollary}
\begin{proof}It simply follows by Proposition \ref{prop:fundamental_heat_solution}, linearity and  \eqref{eq:heat_semigroup_exp}.
\end{proof}

The  solution of the discrete heat initial value problem on  $\Omega_h^+$ with null boundary condition can be obtained by an odd reflection from \eqref{eq:HeatSolution_Omega_h}.

\begin{proposition}
\label{lemma1}
The solution  to the following initial-boundary problem on   $[0,T]\times\Omega_{h,0}^+$  
\begin{equation}\label{eq:heat_equation_Omega_h_+}
	\begin{cases}
		\displaystyle\frac{\partial w}{\partial t}(t,x) = \Delta_h w(t,x) =  \dmen_h \dpiu_h w(t,x), & \quad (t,x)\in  (0,T] \times \Omega_h^+;\\
		w(0,x) = f(x), &\quad x\in   \Omega_h^+;\\
         w(t,0)=0, &\quad t\in   [0,T],
	\end{cases}
\end{equation}
admits the following representation
\begin{equation}\label{eq:HeatSolution_Omega_h_+}
	w(t,x) =  h\sum_{y\in \Omega_h^+} G_D(t,x,y)f(y),
\end{equation}
where the discrete heat kernel $G_D$ on $\Omega^+_h $ is given by
\begin{equation}\label{eq:Gd}
G_D(t,x,y) = H_D(t,x-y) - H_D(t,x+y).
\end{equation}	
\end{proposition}
\begin{proof}
The proof follows the same procedure as in the continuous case, starting from the problem defined on $\Omega_h$ and by constructing an odd extension for the initial condition.
 Let us consider  an odd reflection $g $ of    the initial data $f $  as a function defined for any $x\in \Omega_h$ as follows
\begin{equation}\label{eq:odd_reflection_g} g(x) = 
	\begin{cases}
		f(x), &   x\in \Omega_h^+,\\
		-f(-x), &    x\in \Omega_h^{-},\\
		0, &  x=0.
	\end{cases}
\end{equation}
Then $\tilde{u}(t,x)$  given by \eqref{eq:HeatSolution_Omega_h} is a solution of \eqref{eq:heat_equation_Omega_h}   in $\Omega_h$, with initial condition $g$ given by \eqref{eq:odd_reflection_g}.  By a little algebra we get that for any $(t,x)\in [0,T]\times\Omega_{h,0}^+$ $w(t,x) = \tilde{u}(t,x)|_{x\in \Omega_{h,0}^+}$  satisfies the problem  \eqref{eq:heat_equation_Omega_h_+} and it admits the representation \eqref{eq:HeatSolution_Omega_h_+}-\eqref{eq:Gd}. Indeed, for any $(t,x)\in [0,T]\times\Omega_{h}^+$,
\begin{eqnarray*}
   w(t,x)  & = &\tilde{u}(t,x) = h\sum_{z \in \Omega_h} H_D(t,x-z)g(z) \\& = & h\sum_{z \in \Omega_h^-} H_D(t,x-z)g(z) + h\sum_{z \in \Omega_h^+} H_D(t,x-z)g(z) \\
	& =&h\sum_{z \in \Omega_h^-} -H_D(t,x-z)f(-z) +h\sum_{z \in \Omega_h^+} H_D(t,x-z)f(z) \\
    & = &h\sum_{z \in \Omega_h^+} \left( H_D(t,x-z) - H_D(t,x+z) \right) f(z).
\end{eqnarray*}
Furthermore,  for any $x\in\Omega_{h,0}^+$, $w(0,x)   =  \tilde{u}(0,x)=g(x)=f(x)$ and for any $t\in (0,T]$, $w(t,0) =0,$ for the symmetry of $H_D(t,\cdot)$.
\end{proof}
Finally, we may obtain a representation of the solution of the random initial boundary value problem \eqref{eq:discreteheatequation} for the space-discrete heat equation.
\begin{proposition}\label{prop:solution_random_system}
The function $u$ on $ [0,T]\times\Omega^+_{h,0}$ defined,  as   
\begin{equation}\label{eq:solution_random_heat_equation_HD}
\begin{split}
    u(t,x) &= -\int_0^t (\dmen_{h}+\dpiu_{h}) H_D(t-s,x)\psi(s) ds,  \quad (t,x)\in (0,T]\times \Omega_h^+,\\
    u(t,0)&=\psi(t), \hspace{5.2cm} t\in [0,T],
    \end{split}
\end{equation}
is a solution to  the random heat equation \eqref{eq:discreteheatequation}, with     $\psi$ as in  \eqref{eq:regularity_boundary_process} and such that $ \psi(0)=0$. 
\end{proposition}

\begin{proof}
Since $u(t,x)$ is the solution to  \eqref{eq:discreteheatequation}, let us consider the weak formulation of the solution by taking as a test function, for any $(t,x)\in [0,T]\times\Omega_h^+$,  the kernel $ G_D(t-s,x,\cdot)$, as follows
\begin{align*}
	0 & = \int_0^t \int_{\Omega_h^+} \left[\partial_s G_D(t-s,x,y)+\dpiu_{h,y}\dmen_{h,y} G_D(t-s,x,y)\right]u(s,y)dyds \\
	& = \int_0^t \int_{\Omega_h^+} u(s,y)\partial_s G_D(t-s,x,y)dyds + \int_0^t\int_{\Omega_h^+} u(s,y) \dpiu_{h,y}\dmen_{h,y} G_D(t-s,x,y)dyds\\
	& = I_1 +  I_2
\end{align*}
Let us note that we have  introduced in the difference derivatives a dependence from the space variable to trace that the differences are done along the second spatial argument $y$, while $x$ is fixed.
 By the integration  by parts and for the initial condition in  \eqref{eq:discreteheatequation}, we get
\begin{eqnarray*}
	I_1 & = &\int_0^t \int_{\Omega_h^+} u(s,y)\partial_s G_D(t-s,x,y)dyds \\
    & = &\int_{\Omega_h^+}  \left( u(t,y)G_D(0,x,y) -  u(0,y)G_D(t,x,y) \right)dy -\int_0^t\int_h^{\infty}\partial_s u(s,y)G_D(t-s,x,y)dyds\\
	& =& u(t,x) -\int_0^t\int_{\Omega_h^+} \partial_s u(s,y)G_D(t-s,x,y)dyds.
\end{eqnarray*}
Again by integrating by parts  twice as in \eqref{eq:chain_rule_D-D+}   we get   
\begin{align*}
	I_2 & = \int_0^t\int_{\Omega_h^+} u(s,y) \dmen_{h,y} \dpiu_{h,y} G_D(t-s,x,y)dyds	\\
	& = -\int_0^t \psi(s) \dpiu_{h,y} G_D(t-s,x,0) ds + \int_0^t\int_h^{+\infty} \dmen_{h,y}\dpiu_{h,y} u(s,y)G_D(t-s,x,y)dyds\\
    &=    \int_0^t \psi(s) \left(\dmen_{h,x}+\dpiu_{h,x}\right) H_D(t,x) ds + \int_0^t\int_h^{+\infty} \dmen_{h,y}\dpiu_{h,y} u(s,y)G_D(t-s,x,y)dyds,   
\end{align*}
where the last equality is due to the fact that a straightforward calculation leads to   $$   \dpiu_{h,y} G_D (t,x,y)  =   -\dmen_{h,x}H_D(t,x-y) - \dpiu_{h,x} H_D(t,x+y).
$$
Then,
\begin{eqnarray*}
	0 = I_1+I_2 & =&u(t,x)  -\int_0^t\int_h^{+\infty}\partial_s u(s,y)G_D(t-s,x,y) dyds \\
	& & + \int_0^t \psi(s) \left(\dmen_{h,x}+\dpiu_{h,x}\right) H_D(t,x) ds\\
	& &  + \int_0^t\int_h^{+\infty} \dmen_{h,y}\dpiu_{h,y} u(s,y)G_D(t-s,x,y)dyds \\
	& = &u(t,x) + \int_0^t \psi(s) \left(\dmen_{h,x}+\dpiu_{h,x}\right) H_D(t,x) ds. 
\end{eqnarray*} 
 The thesis is achieved.
\end{proof}

\subsection{A priori bounds for the space-discrete heat equation}

We start this section with a useful proposition on $L^{\infty}$ bounds of the function $u$ solution to equation \eqref{eq:discreteheatequation}.
 
\begin{proposition}\label{proposition:heatbounds}
    Suppose that $u$ is solution to equation \eqref{eq:discreteheatequation}, then 
    \[\| u \|_{L^{\infty}([0,T] \times \Omega_h^+)} \leq \|\psi \|_{L^{\infty}([0,T])}.\]
\end{proposition}
\begin{proof}
By Lemma \ref{lemma:FK}, the unique solution to equation \eqref{eq:discreteheatequation} has the following probabilistic representation 
\[u(t,x)=\mathbb{E}\left[\psi(X^x_{\tau^x}) \mathbb{I}_{\tau^x \leq t} \right],\]
where $X_t^x$ is the stochastic process solving the martingale problem associated with the discrete Laplace operator, introduced in Remark \ref{process_laplacian}, such that $X^x_0=x$, and $\tau^x$ is the hitting time of the process $X^x_t$ with respect to the set $\Omega_h \backslash \Omega_h^+$. The statement follows directly from the monotonicity of expectation.
\end{proof}

Next step is to establish a bound for the norm of $u$ in specific Besov spaces, so that we gain also bounds for the forward and backward derivatives  $\dpiu_h u$ and  $\dmen_h u$. We first need to consider some regularization properties of the discrete heat kernel. The following result allows us to gain  regularity in space of order $2m, m>0$, whenever  the
the heat semigroup is applied  at time $t$,  with a  payoff of order $t^{-m}$ .
  
\begin{lemma}[Heat semigroup regularization]\label{lem:heat_kernel_regu}
    Let  $ p,q \in [1,+\infty]$, $\alpha>0$,  and  $g\in B_{p,q}^{\alpha}(\Omega_h)$. Then, if $m>0$, for every $t>0$
    \begin{equation}\label{eq:semigroup_estimate_besov}
      \left\|e^{\Delta_h t}g \right\|_{B_{p,q}^{\alpha+2m}(\Omega_h)}  \le  C t^{-m} \left\|g\right\|_{B_{p,q}^\alpha(\Omega_h)},
    \end{equation}
  where  $e^{\Delta_h t}$ denotes the discrete heat semigroup  \eqref{eq:heat_semigroup_exp}.
\end{lemma}
\begin{proof}
        The semi-discrete case on $B_{p,q}^{s}(\Omega_h)$ can be proven by following the same line of the continuous case $B_{p,q}^{s}(\R^d)$ in \cite{devecchi_nicolay_2021elliptic,2022_Turra}, in particular, by using Lemma 2.10 in \cite{devecchi_nicolay_2021elliptic}.
\end{proof}
 
\begin{lemma}[Dirac-delta semigroup regularization]\label{lem:stime_delta_dirac_semigroup_regularization}
    Let  $p\in[1,+\infty]$, $k>1$ and $ \alpha<1/p-1$. Then, for every $t,\tau >0$ and $\theta \in (0,1)$, the following estimate holds
    \begin{equation}\label{eq:lemma_semigroup _delta_dirac}
        \left\|\int_t^{t+\tau} e^{\Delta_h s}\delta_0ds\right\|_{B^{\alpha+k+1}_{p,p}(\Omega_h^+)} \le C \,\tau^\theta\,  t^{-\left( {\theta\frac{k+1}{2} + (1-\theta)\frac{k-1}{2}}\right)}.
    \end{equation}
\end{lemma}
\begin{proof}
 Note that   since $e^{\Delta_h s} \delta_0$ is an even function,  the Besov norms on  $\Omega_h$ and on   $\Omega_h^+$, respectively, are equivalent. So from \eqref{eq:semigroup_estimate_besov}, with $m=(k+1)/2$ 
   from   Proposition \ref{prop:Delta_in_Besov} $\delta_0 \in B^\alpha_{p,p}(\Omega_h^+), $ we get
\begin{align*}
    \left\| \int_t^{t+\tau} e^{\Delta_h s}\delta_0 \, ds\right\|_{B^{\alpha+k+1}_{p,p}(\Omega_h^+)} & \le \int_t^{t+\tau} ||e^{\Delta_h s} \delta_0 ||_{B^{\alpha+k+1}_{p,p}(\Omega_h^+)} ds {\le} \,   C  \int_t^{t+\tau} {s^{-\frac{k+1}{2}}} ds. 
\end{align*}
Since the integral is less than both $\tau t^{-(k+1)/2}$ and, for $k>1$, $ \frac{2}{k-1}t^{-(k-1)/2}$, by a geometric interpolation we may write, for $\theta \in (0,1)$
\begin{eqnarray*}
    \int_t^{t+\tau} s^{-\frac{k+1}{2}} ds &\le&  \left(\tau t^{-(k+1)/2} \right)^\theta \left(  \frac{2}{k-1}t^{-(k-1)/2}\right)^{1-\theta}\le  C \,\tau^\theta\,  t^{-\left( {\theta\frac{k+1}{2} + (1-\theta)\frac{k-1}{2}}\right)} 
\end{eqnarray*}
  and by inserting the latter estimate in the previous expression we get the claim.
\end{proof}
\begin{lemma}[Forward difference Dirac-delta semigroup regularization]\label{lem:forward_difference_dirac_besov}
        Let  $p\in[1,+\infty]$ and $ \beta\in (1/4,1/2)$.  Furthermore, let $ \alpha<1/p-1$ and $k\in (1,3/2)$.   Then, for any $t,s>0$, \begin{equation}\label{eq:erence_dirac_besov}
        \begin{split}
        \dpiu_h e^{\Delta (t-\cdot)}\delta_0,  \dmen_h e^{\Delta (t-\cdot)}\delta_0 &\in B_{1,1}^{-\beta}([0,t],B^{\alpha+k}_{p,p}(\R_+)); \\
        \int_0^s \dpiu_h e^{\Delta_h (t-\tau)}\delta_0 d\tau,\int_0^s \dmen_h e^{\Delta_h (t-\tau)}\delta_0 d\tau &\in B_{1,1}^{1-\beta}([0,t],B^{\alpha+k}_{p,p}(\Omega_h^+)).
        \end{split}
        \end{equation}
\end{lemma}
\begin{proof}
 Let us prove the second  in \eqref{eq:erence_dirac_besov}, only for $\dpiu_h,$.
Let us   fix $t>0$ and denote, for any $s\in(0,t)$ the function $$F(s,\cdot) = \int_0^s \dpiu_h e^{\Delta_h (t-\tau)}\delta_0 d\tau.$$
By the difference characterization of time-space Besov spaces \eqref{eq:difference_characterization_besov} we may estimate its Besov norm as 
\begin{align*}
    ||F||_{B_{1,1}^{1-\beta}([0,t], B_{p,p}^{\alpha+k}(\Omega_h^+))} 
    & = \int_0^t||F(s,\cdot)||_{B^{\alpha+k}_{p,p}}ds +    \int_0^t\int_{|\bar{s}|<1} \frac{||F(s+\bar{s})-F(s)||_{B^{\alpha+k}_{p,p}}}{\bar{s}^{2-\beta}}d\bar{s} ds.
\end{align*}
We start by giving an upper bound for $\|F\|_{L^1([0,t], B_{p,p}^{\alpha+k})}$. For \eqref{eq:semigroup_estimate_besov} and Proposition \eqref{prop:Delta_in_Besov}, because of the bound $\alpha<\frac{1}{p}-1$ we have
\begin{equation*}
\begin{aligned}
    \|F\|_{L^1([0,t], B_{p,p}^{\alpha+k} (\Omega_h^+))} & = \int_0^t\left\|\int_0^s \dpiu_h e^{\Delta_h (t-\tau)}\delta_0d\tau \right\|_{ B_{p,p}^{\alpha+k}}ds \\
& {\leq }C\int_0^t\int_0^s\|e^{\Delta_h (t-\tau)}\delta_0\|_{B^{\alpha+k+1}_{p,p}}d\tau ds
\leq C^\prime \int_0^t\int_0^s{\frac{1}{(t-\tau)^{\frac{k+1}{2}}}}d\tau ds.
\end{aligned}
\end{equation*}

The last integral is bounded, uniformly in $t\in[0,T]$, whenever  
$k<3$. Under the latter condition and hypothesis $\alpha<\frac{1}{p}-1$, we get 
\[
\|F\|_{L^1([0,t], B_{p,p}^{\alpha+k}(\Omega_h^+))} < C_1.
\]
for some $C_1>0$,  uniformly in $t\in[0,T]$.
Applying Lemma \ref{lem:stime_delta_dirac_semigroup_regularization} to the difference norm, we get
\begin{align*}
    ||F(s+\bar{s},\cdot)-F(s,\cdot)||_{B^{\alpha+k}_{p,p}(\Omega_h^+)} & \le {C}\int_s^{s+\bar{s}}\frac{1}{\tau^{\frac{1+k}{2}}}d\tau 
     \le \frac{1}{s^{\frac{k+1}{2}\theta+(1-\theta)\frac{k-1}{2}}} \bar{s}^\theta.
\end{align*}
Now, putting all the estimates together, we finally obtain:
\begin{equation}\label{eq:stima_F_besov}
    ||F||_{B_{1,1}^{1-\beta}([0,t], B_{p,p}^{\alpha+k}(\Omega_h^+))} \le C_1 + C_2 \int_0^t \int_{|\bar{s}|<1} \frac{1}{s^{\frac{k+1}{2}\theta+(1-\theta)\frac{k-1}{2}}} \sigma^{\theta+\beta-2} d\bar{s} ds
\end{equation}
For the integrability of the r.h.s.term we should ensure the following conditions:
\begin{eqnarray}
    \theta+\beta -2 > -1 & \rightarrow & \beta+\theta >1 
    \label{eq:summability_1}\\
    \frac{k+1}{2}\theta + (1-\theta)\frac{k-1}{2} < 1 & \rightarrow & \theta< \frac{3-k}{2} \label{eq:summability_3}
\end{eqnarray} 
From \eqref{eq:summability_1}
 the following condition on time regularity $\beta$ holds:
\begin{equation}\label{eq:cond_beta}
        \beta > 1-\theta
\end{equation}
Combining \eqref{eq:cond_beta} with \eqref{eq:summability_3}, we find the
 connection between the time regularity $\beta$ and spatial regularity $k$:
\begin{equation}\label{eq:legame_beta_kappa}
\beta > \frac{k-1}{2}
\end{equation}
From the lower bound in \eqref{eq:legame_beta_kappa}, we should impose that
\begin{equation*}
    \frac14 \ge \frac{k-1}{2} \,\rightsquigarrow \, k \le \frac32,
\end{equation*} 
that is the assumed bound for $k$. Hence, the integrability of the right hand side in \eqref{eq:stima_F_besov} is obtained and then we get  the thesis.
\end{proof}
\begin{remark}
For the bounds upon the parameters   in Lemma \eqref{lem:forward_difference_dirac_besov},   the maximum space regularity in   $B_{p,p}^{\alpha+k}$,  is obtained whenever
\begin{equation*}
  r<  \alpha+k \le \frac1p -1 +\frac32 = \frac1p + \frac12;
\end{equation*}
hence, by choosing $p=2$, we gain $\alpha+k=1$ regularity in space.
\end{remark}
We are ready to provide bounds   to the solution $u$  of the system \eqref{eq:discreteheatequation} in the Sobolev space, uniformly in time and for specific parameters. A continuity result with respect to the boundary condition is given.
 
\begin{proposition}[Continuity w.r.t. to the boundary]  \label{proposition:estimate_u}
    Let $u$ be the solution to the semi-discrete random heat problem \eqref{eq:discreteheatequation}.
    Let us suppose that $\psi$ satisfies \eqref{eq:regularity_boundary_process} with   $\psi(0)=0$; in particular $\beta\in (1/4,1/2)$. Then, for every $p\in[1,\infty) $ and $\alpha<1/p-1,$ there exist constants $C, C'$ not depending on $\psi$ such that, for any $r< 2\beta+\/p$
    \begin{equation*}
        ||u||_{L^\infty(\R_+,W^{p,r}(\Omega_h^+))} \le C' ||\psi||_{C^\beta([0,T])}.
    \end{equation*}
    In particular, we get
    \begin{equation}\label{eq:continuity_u_Psi_H1}
        ||u||_{L^\infty(\R_+, H^1(\Omega_h^+))} \le C ||\psi||_{C^\beta ([0,T])}.
    \end{equation}
\end{proposition}
 
\begin{proof} 
From the  representation of  $u$  given in \eqref{eq:solution_random_heat_equation_HD}, we move  some regularity from the boundary function $\psi$ to the spatial variable. Let us consider
\begin{align*}
    u(t,x) & = -\int_0^t \left(\dmen_h +\dpiu_h\right) H^D(t-s,x)\psi(s)ds     = - \int_0^t  \left(\dmen_h +\dpiu_h\right) e^{\Delta_h(t-s)} \delta_0(x)\psi(s)ds.
\end{align*}
Since the  boundary function $\psi(s)\in C^\beta ([0,t])= B_{\infty,\infty}^\beta ([0,t]) $, then the product $\dpiu_h e^{\Delta_h (t-s)}\delta_0 \psi(s)$, as a distribution depending on time $s$, is a well-defined distribution. This allows us to derive the following estimate for the Besov norm of the time integral as
\begin{equation*}\label{eq:stima_besov_heat_solution}
\left\|\int_0^t \dpiu_h e^{\Delta_h(t-s)} \delta_0 (x)\psi(s)ds \right\|_{B^{\alpha+k}_{p,p}} \leq \| \dpiu_h e^{\Delta_h (t-\cdot)}\delta_0\|_{ B_{1,1}^{-\beta}([0,t],B^{\alpha+k}_{p,p})} \|\psi(\cdot) \|_{C^\beta ([0,t])},
\end{equation*}
which is bounded for Lemma \ref{lem:forward_difference_dirac_besov} and hypotheses \eqref{eq:regularity_boundary_process}. This
implies that, whenever $r<\alpha+k$, the $L_t^\infty W_x^{p,r}$ norm of $u$ is bounded (see, e.g., \cite{devecchi_nicolay_2021elliptic}, Theorem 2.18) and the thesis is proven.
\end{proof}

\begin{proposition}
    Let $u$ be the solution of the random space-discrete heat equation \eqref{eq:discreteheatequation}. Then the following estimate holds:
    \begin{equation}\label{eq:continuity_D_hu_Psi_H1}
        \|\dpiu_h u \|_{L_t^\infty L_x^2} \le C_1' \|\psi\|_{C^\beta([0,T])}
    \end{equation}
\end{proposition}
\begin{proof}
    The estimation  derives from  a direct application of Proposition \ref{proposition:estimate_u} with $p=2, r=1.$
\end{proof}

\begin{remark}
An alternative way to prove that the maximum regularity that could be gained in space is 1 is to fix $\beta$ in \eqref{eq:legame_beta_kappa} and derive $k$ in terms of $\beta$, obtaining $ k < 2 (\beta +\frac12) = 2\beta+ 1 $.

The maximum regularity we can gain in space is given by $  \alpha + 2\beta+ 1 > 1$, implying $ \beta  > \frac{p-1}{2p}. $
Again, by choosing $p=2$ we get $\beta> \frac14$, which is the lowest regularity admissible for the boundary condition $\psi$.
\end{remark}

\section{A priori estimates for a quasi-linear equation}\label{section:apriori}

 Now we turn our attention to the nonlinear system \eqref{eq:nonlinear_discrete_system_v} for the function $v$.

 \smallskip
Following \cite{2025_SPA_MauMorUgo,2005_GN_NLA}, we perform a sort of linearization of the random semi-discrete system \eqref{eq:system_s_c_discrete}.   Let  $f:[0,T]\times \Omega_h^+ \to\mathbb{R}$ be a Borel function satisfying the following conditions for some $K>0$.
\begin{align}
\begin{aligned}\label{eq:hp_f}
&f\in C([0,T], L^2(\Omega_h^+)) \cap L^2([0,T],W^{1,2}(\Omega_h^+)); \\
    &\|f\|_{C([0,T],L^2(\Omega_h^+))}^2+\|D^+_x f\|^2_{L^2([0,T],L^2(\Omega_h^+))}\le K;\\
&f\ge 0;\\
&f(\cdot,0)=\psi; \\
&\|f\|_{L^\infty([0,T],L^\infty(\Omega_h^+))}\le \eta .
\end{aligned}
\end{align}
Given,  $f:[0,T]\times \Omega_h^+ \to\mathbb{R}$ as in \eqref{eq:hp_f}, let us  consider the following "linearized version" of the semi-discrete system \eqref{eq:system_s_c_discrete}
\begin{equation}\label{eq:lin_PDE_sg_s}
\begin{split}
\partial_t \tilde{s} & = \Delta_h \tilde{s} + b_g(t,x)[\dpiu_h g\dpiu_h \tilde{s} + \dmen_h g \dmen_h \tilde{s}]  + \gamma_g \tilde{s} (Bf-1); \\
\partial_t g & = -\lambda f\varphi(g) g,  
\end{split}
\end{equation}
 where $b_g, \gamma_g$ are the functions in \eqref{eq:functions_bc_gamma} defined on $g$, and where $g$ admits the following explicit expression 
\begin{equation}\label{eq:explicit_g}
    g(t,x) =  {Ac_0(x)}\left[{\varphi(c_0(x))e^{\lambda A\int_0^t f(\tau,x)d\tau}-B c_0(x)}\right]^{-1}.
\end{equation}
Equivalently, for the splitted variable $\tilde{v}$, such that   $\tilde{s}=u+\tilde{v}$, where  $u$ is the solution of the initial boundary value problem for the system \eqref{eq:discreteheatequation}, we have 
\begin{equation}\label{eq:linearized_v} 
\begin{split}
    \partial_t \tilde{v} & = \Delta_h \tilde{v} + b_g[\dpiu_h g \dpiu_h \tilde{v} + \dmen_h g \dmen_h \tilde{v}] -\gamma_g(1-Bf)\tilde{v} \\
    & \quad + b_g [\dpiu_h g \dpiu_h u+\dmen_h g \dmen_h u] - \gamma_g(1-Bf) u, \quad (t,x)\in (0,T]\times \Omega_h^+;    \\
    \tilde{v}(0,x) & = s_0(x), \hspace{6.2cm}   x\in  \Omega_h^+; \\
    \tilde{v}(t,0) &= 0, \hspace{6.8cm}  t\in [0,T]. 
\end{split}
\end{equation} 
The solution $\tilde{v}$ of \eqref{eq:linearized_v} is obtained, once $u$, its  difference derivatives $D^{\pm}_h u$ and the function  $f$ as in \eqref{eq:hp_f} are given.

\smallskip

In the following,  some more regularity hypotheses  are needed to achieve an a priori bounds for $v$ and then for $s$.

\begin{assumptions}\label{assumption:_bound_s0_c0_varphi}
Let us assume the following  hypotheses upon the initial and boundary conditions. The boundary condition $\psi$ satisfies \eqref{eq:regularity_boundary_process}  and  $\psi(0)=0$. Furthermore, 
\begin{eqnarray} 
    &    0\leq s_0(x) \leq \eta,          \,\, 
        s_0(0)=0,  & s_0 \in W^{1,p}(\Omega_h^+),  \,\, 2 \leq p \leq \infty;\label{eq:assumption_bound_s0_c0_varphi_1}\\
        &0<c_m\leq c_0(x) \leq C_0, \,\,\,&
     C_0- c_0 \in  W^{1,2}(\Omega_h^+). \label{eq:assumption_bound_s0_c0_varphi_2}
     \end{eqnarray}
     Furthermore, the bound \eqref{eq:limit_porosity}  for the porosity $\varphi$ holds and when $B>0$ then $B<1/\eta.$
\end{assumptions}

\begin{proposition}\label{proposition:estimate_g}
Let us suppose that assumptions \eqref{eq:hp_f}, \eqref{eq:assumption_bound_s0_c0_varphi_1} and \eqref{eq:assumption_bound_s0_c0_varphi_2} are satisfied.  
Then, the solution $g$ to \eqref{eq:explicit_g} and its derivative admit the following estimates
\begin{equation}\label{eq:dpiu_g_estimates}
\begin{split}
    \|C_0- g\|_{L^\infty_t L^2_x}^2 &\le k \left(\|C_0-c_0\|_{L^2(\Omega_h^+)}^2 + \kappa \| f\|^2_{C([0,T],L^2(\Omega_h^+))}\right);\\ \\
    \| g\|_{L^\infty([0,T]\times\Omega_h^+)} & \le \hat{C}_1(C_0,A,B)\\
        \left\|\dpiu_h g\right\|_{L^{\infty}_t L^2_x}^2 &\le\hat{C}_2(C_0,A,B,) T \left(\|\dpiu_h c_0\|_{L^2(\Omega_h^+)}^2 + \kappa \|\dpiu_h f\|^2_{L_t^2 L^2_x}\right).
    \end{split}
\end{equation}
where $\hat{C}_1,\hat{C}_2$ are positive measurable functions increasing with respect to all their variables. The last estimate is extended to $\|\dpiu_h\varphi(g)\|_{L^\infty_t L^2_x}$ up to a constant factor $|B|$.
\end{proposition}
    \begin{proof} The estimate for $g$ can be simply obtained with the same approach as the continuous counterpart in \cite[ Proposition 32]{2025_SPA_MauMorUgo}.\\
     For the second bound in  \eqref{eq:dpiu_g_estimates}, we introduce the following notations, for any $t\in(0,T]$ and $x\in \Omega_h^+.$
    \begin{equation*}
        g(t,x) = \frac{Ac_0(x)}{\zeta(t,x)},\quad\zeta(t,x) := \varphi(c_0(x))\tilde{h}(t,x)-Bc_0(x), \quad \tilde{h}(t,x):=e^{\lambda A \int_0^t f(\tau,x)d\tau} .
    \end{equation*}
    We note that, when $A=1$ and $B=1$, we have
    \[\zeta(t,x) = A \tilde{h}(t,x) + B c_0(x)(\tilde{h}(t,x)-1) \geq  A, \]
    where we used that $\tilde{h} \geq 1$ and $c_0 \geq c_m$. Furthermore for $A=1$ and $B=-1$ we have
    \[\zeta(t,x) \geq \tilde{h}(A+ B c_0(x)) \geq (A+B C_0) >0  \]
    by the requirement on $C_0$. Since $A c_0(x) \leq A C_0$, the second inequality of \eqref{eq:dpiu_g_estimates} follows. \\
    
    For the third bound in \eqref{eq:dpiu_g_estimates}, by Leibniz rule (Lemma \eqref{lem:Leibniz}) we get
    \begin{align}
        \dpiu_h g(t,x) &  =\frac{\zeta(t,x) \dpiu_h (A c_0(x)) - A c_0(x)\dpiu_h\zeta(t,x) }{(\tau_h\zeta(t,x))\cdot\zeta(t,x)}, \label{eq:dpiu_g}
    \end{align}
    where
    \begin{align*}
       \dpiu_h\zeta(t,x)&= \varphi(c_0(x)\dpiu_h \tilde{h}(t,x)+\tilde{h}(x+h,t)\dpiu_h \varphi(c_0(x))-B\dpiu_h c_0(x),\\
      \dpiu_h \tilde{h}(t,x) & =\frac1h \left(e^{\lambda A \int_0^t f(x+h,\tau)d\tau} -e^{\lambda A \int_0^t f(\tau,x)d\tau}\right)\\
       &= \int_0^1 \left(e^{\lambda A \int_0^t   f(x,\tau) +\theta h\dpiu_h f(x,\tau)  d\tau}\right) \left(\lambda A \int_0^t  \dpiu_h f(x,\tau)  d\tau\right)d\theta.
    \end{align*}
The bound for    $\dpiu_h \tilde h$ is trivially  given by
\begin{equation*}
\begin{aligned}
         \|\dpiu_h \tilde h \|_{L^{\infty}_t L_x^2}^2 &
        & \le \exp{\left(C_1 T\left\|f\right\|_{L^{\infty}_t L^\infty_x}\right)}
        \left(\lambda A T^2\left\| \dpiu_hf\right\|^2_{L^2_t L^2_x}\right),\\
\end{aligned}
\end{equation*}
where we used that $\|h D_{h}^{+}f \|_{L^{\infty}_tL^{\infty}_x}\leq 2 \| f\|_{L^{\infty}_tL^{\infty}_x}$, uniformly in $h>0$, and that
\[\int_0^T\left(\int_0^t D^+_hf(x,\tau)d\tau\right)^2 d t \leq \frac{T^2}{2} \int_0^T|D^+_hf(x,\tau)|^2 d \tau. \]
Furthermore,
\begin{align*}
    \|\zeta\|_{L^\infty_t L^2_x}^2 &=\sup_{t\in [0,T]} \| \varphi(c_0) e^{\lambda A \int_0^t f(\tau, x)d\tau} - Bc_0\|_{L^2(\Omega_h^+)}^2\\
    & \leq \varphi_{max} \exp{\left(\lambda A T \|f\|^2_{L^\infty_{t}L^\infty_x}\right)}+ |B|C_0 \leq C
\end{align*}
   and
    \begin{align*}
        & \left\|\dpiu_h \zeta\right\|^2_{L^\infty_t L^2_x} = \sup_{t\in [0,T]}\left\|\dpiu_h \varphi(c_0) \tilde{h} + \varphi(c_0) \dpiu_h \tilde h - B\dpiu_h c_0\right\|^2_{L^2(\Omega_h^+)} \\
         & \le C_1\|\dpiu_h c_0\|_{L^2(\Omega_h^+)}^2 + C_2\|\dpiu_h f\|^2_{L^2_{t}L^2_x}
    \end{align*}
   where we have exploited the hypothesis on the data $f$.

Since we have proved that $\zeta(t,x)>0$, going back to the definition of $\dpiu_h g$ in \eqref{eq:dpiu_g} we get
\begin{align*}
    \|\dpiu_h g\|_{L^\infty_t L^2_x}
    \le C\|\zeta(t,x) \dpiu_h (A c_0(x)) - A c_0(x)\dpiu_h\zeta(t,x)\|_{L^\infty_t L^2_x}
\end{align*}
By rearranging all the estimates combined with the hypothesis for f in \eqref{eq:hp_f} we obtain the following estimation on $\dpiu_h g$
\begin{equation*}
    \|\dpiu_h g\|_{L^\infty_t L^2_x}^2 \le c T\left( \|\dpiu_h c_0\|_{L^2(\Omega_h^+)}^2 + \kappa \|\dpiu_h f\|^2_{L^2_{t}L^2_x}\right).
\end{equation*}
The estimate for the  term $\|\dpiu_h\varphi(g)\|_{L^\infty_t L^2_x}$ are related to the previous estimates for $g$ since
\begin{align*}
    \|\dpiu_h\varphi(g)\|_{L_t^\infty L_x^2}^2 & = \sup_{t\in [0,T]} \|\dpiu_h\varphi(g)\|^2_{L^2(\Omega_h^+)}
     \le \sup_{t\in [0,T]} |B| \|\dpiu_h g\|_{L^2(\Omega_h^+)}^2
\end{align*}
\end{proof}

\begin{proposition}\label{proposition:sbounds}
Let us suppose that Assumptions \ref{assumption:_bound_s0_c0_varphi} hold and that $f$ satisfy the conditions \eqref{eq:hp_f}. Then there is a unique solution $\tilde{s}$ to the equation \eqref{eq:lin_PDE_sg_s} and we have
\[0 \leq \tilde{s}(t,x) \leq \eta \]
for any $(t,x) \in [0,T] \times \Omega_h^+$
\end{proposition}
\begin{proof}
 In Appendix \ref{appendix:proof}, for the convenience of the reader, the Feynman-Kac representation formula and other related results for a general continuous time process taking values on a space lattice are derived. By using those results (see  Appendix \ref{appendix:proof}),
we are going to prove that the solution $\tilde{s}$ to equation \eqref{eq:lin_PDE_sg_s} has the following stochastic representation:
\begin{equation}\label{eq:tildesrepresentation1}
\tilde{s}(t,x)=\mathbb{E}\left[e^{\int_{T-t}^{\tau^{t,x}}\Gamma_{g,f}(T-\sigma,X^{T-t,x}_{\sigma}) d \sigma}\left(s_0(X_{T}^{T-t,x})\mathbb{I}_{\tau^{t,x} > T} +\psi(T- \tau^{t,x})\mathbb{I}_{\tau^{t,x} \leq T}\right) \right]
\end{equation}
where $X^{\tau,x}_{\cdot}$ is the solution to the martingale problem associated with the operator 
\[\bar{L}=\partial_t + \Delta_h + \beta_g^{+}(T-t,x) D_h^+ -\beta_g^{-}(T-t,x) D_h^-,\]
where the function $\Gamma_{f,g}$ and $\beta^{\pm}_{g}$ are given by the expressions
\begin{align*}
    \beta_{g}^{\pm}(t,x)=&\pm\frac{B D^{\pm}_hg(t,x)}{2(A+B g(t,x))} ,\\
    \Gamma_{g,f}(t,x)=&\lambda g(t,x) (B f(t,x)-1).
\end{align*}
Since $s_0,\psi$ satisfy $0 \leq s_0, \psi \leq \eta$, and $\Gamma_{g,f} \leq 0$ (by the hypothesis $B<\frac{1}{\eta}$) and thus $ e^{\int_0^{\tau^{t,x}}\Gamma_{g,f}(T-\sigma,X^{T-t,x}_{\sigma}) d \sigma} \leq 1$, the thesis follows from  representation \eqref{eq:tildesrepresentation1}. Formula \eqref{eq:tildesrepresentation1} is a particular case of the general formula (Lemma \ref{FKformula}) in Appendix \ref{appendix:proof}.  \\
In order to prove equation \eqref{eq:tildesrepresentation1}, it is convenient to consider the function $\hat{s}(t,x)=\tilde{s}(T-t,x)$ and prove an analogous representation for the function $\hat{s}$. If the function $\tilde{s}$ satisfies equation \eqref{eq:lin_PDE_sg_s}, then the function $\hat{s}$ satisfies the following equation 
\begin{equation}\label{eq:hats1}
\begin{split}
&\bar{L}(\hat{s})(t,x)+\Gamma_{g,f}(t,x)\hat{s}(t,x)= 0\\
&\hat{s}(T,x)=s_0(x)\\
&\hat{s}(t,0)=\psi(T-t).
\end{split}
\end{equation}
Equation \eqref{eq:hats1} is equivalent to equation \eqref{eq:discretegeneric1} in Appendix \ref{appendix:proof} in the case where 
\[C^{\pm}(t,x)=\frac{1}{2h} \left( 2- \beta_g^{\pm}(T-t,x) \right)= \frac{1}{2h} \left(2-\frac{B(g(T-t,x)-g(T-t,x\pm h))}{A+B g(T-t,x)} \right)  \]
and $V(t,x)=- \Gamma_{g,f}(T-t,x)$. Both the functions are positive and bounded thanks to the conditions we have imposed on the functions $g$, $f$ and $B$.
\end{proof}
\begin{corollary}\label{corollary:vbounds}
 Suppose that Assumptions \ref{assumption:_bound_s0_c0_varphi} hold and that $f$ satisfies the conditions \eqref{eq:hp_f}. Then
 \[\|\tilde{v} \|_{L^{\infty}([0,T]\times \Omega_h^+)} \leq 2 \eta. \]
\end{corollary}
\begin{proof}
Since $\tilde{v}=\tilde{s}-u$ the result follows from Proposition \ref{proposition:heatbounds} and Proposition \ref{proposition:sbounds}.
\end{proof}

\begin{proposition}[Uniform $L^2$ estimate for $\tilde{v}$]\label{proposition:estimates_v}
  Suppose the  hypotheses in Proposition \ref{proposition:estimate_g} hold. Then,  the solution $\tilde{v}$ to the linearized system \eqref{eq:linearized_v} satisfies the following bound 
  \begin{equation}\label{eq:stima_linearized_v}
      \sup_{t\in [0,T]}\int_{\Omega_h^+}  \tilde{v}^2(t,x) dx + \int_0^T\int_{\Omega_h^+}[(\dmen_h \tilde{v})^2 + (\dpiu_h \tilde{v})^2)] dx  ds  \leq \frac{1}{\varepsilon}  \bar{C} +     \varepsilon\|\dpiu_h f\|^2_{L_t^2 L_x^2},
  \end{equation}
 where $\bar{C}$ is an (increasing and measurable) function depending upon $\eta, \lambda, C_0,\varphi_{min}, B,T,  \|\psi\|_{C^\beta_t}  $ in Assumptions \eqref{assumption:_bound_s0_c0_varphi} and in conditions \eqref{eq:hp_f}  but not on $K$ in \eqref{eq:hp_f}, and  $ \varepsilon \leq 2\varphi_{min}$.
\end{proposition}
\begin{proof}
From \eqref{eq:lin_PDE_sg_s} and \eqref{eq:linearized_v}, with a little algebra,  we obtain   
\begin{align*}
    \partial_t(\varphi \tilde{v}^2)&= 2\varphi v \partial_t \tilde{v} + \tilde{v}^2\partial_t\varphi(g)\\
    &=  \tilde{v}[\dpiu_h(\dmen_h \tilde{v})\varphi)+\dmen_h((\dpiu_h \tilde{v})\varphi)] - 2\lambda\varphi(g) g (1-Bf)\tilde{v}^2 \\
    &\quad +\tilde{v}B(\dpiu_h g \dpiu_h u+\dmen_h g\dmen_h u)-2\lambda\varphi(g) g(1-Bf)u\tilde{v}- \lambda B\varphi(g) g f \tilde{v}^2).
\end{align*} 
By integrating in time and space, we observe that  
\begin{equation*}
\begin{aligned}
  \int_{\Omega_h^+}\int_0^t \tilde{v}[\dpiu_h((\dmen_h \tilde{v})\varphi)+\dmen_h((\dpiu_h \tilde{v})\varphi)] dxds&     = -\int_{\Omega_h^+} \int_0^t [(\dmen_h \tilde{v})^2 + (\dpiu_h \tilde{v})^2)]\varphi dx ds.
\end{aligned}
\end{equation*}
Hence, we get the following equality
\begin{align*}
&  \int_{\Omega_h^+} \varphi(g)\tilde{v}^2(t,x) dx  - \int_{\Omega_h^+} \varphi(g(0,x))(\tilde{v}(0,x))^2 dx +\int_{\Omega_h^+} \int_0^t [(\dmen_h \tilde{v})^2 + (\dpiu_h \tilde{v})^2)]\varphi dx ds \\
    & = -\int_{\Omega_h^+}\int_0^t\lambda \varphi(g) g (2-Bf) \tilde{v}^2 dx ds +\int_{\Omega_h^+}\int_0^t (\dpiu_h u \dpiu_h\varphi+\dmen_h u \dmen_h \varphi)\tilde{v}\, dxds\\ 
    & \quad - \int_{\Omega_h^+}\int_0^t 2\lambda\varphi(g)g(1-Bf)\tilde{v} u\, dxds\\
    &= R_1+R_2+R_3=R_1+\left(R_2^+ +R_2^-\right)+R_3
\end{align*}
From \eqref{eq:limit_porosity} and \eqref{eq:hp_f}, we can trivially show that $R_1$ is negative. Indeed, when $B=1$ we have $2-Bf=2-f>1-f>0$ (since $f\le \eta \le 1$) and when $B=-1$ we get $2-Bf=2+f>0.$\\
 As the term $R_2=R_2^+ +R_2^-$ regards, we detail the bound only for $R_2^+$, since the $R_2^-$ term can be treat   analogously.   
By the generalized H\"older inequality    $\left(\frac{1}{1}=\frac{1}{2}+\frac{1}{2}+\frac{1}{\infty}\right), $ by \eqref{eq:continuity_u_Psi_H1} and \eqref{eq:continuity_D_hu_Psi_H1}, we get  
\begin{align*}
	|R_2^+ | & = \left|\int_{\Omega_h^+} \int_0^t (\dpiu_h u\dpiu_h \varphi)\tilde{v} dr ds \right| \leq 2 T \|\dpiu_h u\|_{L^\infty_t L_x^2} \,\|\dpiu_h \varphi\|_{L_t^\infty L_x^2}\, \|\tilde{v}\|_{L_t^{\infty}  L_x^\infty} \\
 & \le 4 \eta T |B| \|\dpiu_h u\|_{L_t^\infty L_x^2} \|\dpiu_h \varphi\|_{L_t^\infty L_x^2}\\ 
 & \le 4 c \eta |B|  T^2  \|\dpiu_h u\|_{L^\infty_t L_x^2}  \left(\|\dpiu_h c_0\|_{L^2(\Omega_h^+)} + \kappa \|\dpiu_h f\|_{L_t^2 L^2_x}\right)\\
 & \le 4 c \eta |B|  T^2  \|\dpiu_h u\|_{L^\infty_t L_x^2}\|\dpiu_h c_0\|_{L^2(\Omega_h^+)}+ \frac{4 c^2 \kappa^2 \eta^2 |B|^2  T^4}{\varepsilon'}  \|\dpiu_h u\|_{L^\infty_t L_x^2}^2 \\
 &+\varepsilon' \|\dpiu_h f\|_{L_t^2 L^2_x}^2   \le\frac{1}{\varepsilon'} C_2(\eta,\|\psi\|_{C^{\beta}([0,T])},B,T) + \varepsilon' \|\dpiu_h f\|_{L_t^2 L^2_x}^2,
\end{align*}
 where  with $C_{1}$ is a measurable increasing functions of all its variables, and $\varepsilon'>0$. 

Finally for the $R_3$ term, we apply Holder inequality with $p=1,q=\infty$ and we  carry  $||\tilde{v}||_{L^\infty_{t,x}}$ in the interpolated space $L^2([0,T], H^1)$ as follows:
\begin{align*}
	|R_3| &  = \left |\int_0^t\int_{\Omega_h^+} - 2 \lambda \varphi(g) g u (1-Bf)\tilde{v} \, dxdr\right| \\
 & \le C\lambda T ||g||_{L_t^\infty L_x^\infty}(A+|B|||g||_{L_t^\infty L_x^\infty}) ||\tilde{v}||_{L_t^{\infty} L^\infty_x} ||u||_{L^\infty_t L^1_x}\\
	& \le C\lambda T \hat{C}_1(C_0,A,B) (A+|B|\hat{C}_1(C_0,A,B)) \eta  ||u||_{L^\infty_t L^1_x} \\
    &\leq \frac{1}{\varepsilon'} C_{3}(\lambda,T,C_0,A,B,\eta,\|\psi \|_{C^{\beta}([0,T])}) 
\end{align*}
where  $C_{3}$ is a measurable increasing function and $\epsilon'\leq 1$.
 
Hence, by means of the estimates of the terms $R_1$-$R_3$ and by considering the lower bound in \eqref{eq:limit_porosity}  we get a discrete $L^2-$ bound for $\tilde{v}$ and its (symmetric) difference derivative
\begin{equation*}
\begin{aligned}
  \int_{\Omega_h^+} \tilde{v}_t^2 dx + & \int_{\Omega_h^+} \int_0^T  [(\dpiu_h \tilde{v})^2 + (\dmen_h \tilde{v})^2]  \le \\
    & \le \int_{\Omega_h^+} \frac{\varphi(g_t)}{\varphi_{min}} \tilde{v}_t^2 dx + \int_{\Omega_h^+}\int_0^t \frac{\varphi(g)}{\varphi_{min}} \left[(\dmen_h \tilde{v} )^2 + (\dpiu_h \tilde{v})^2\right] dr dx \\
    & \le\frac{2 C_2 +C_3}{\epsilon' \varphi_{min}} +\frac{2\varepsilon'}{\varphi_{min}}  \|\dpiu_h f\|^2_{L_t^2 L_x^2}.\end{aligned}
\end{equation*}
By writing $\varepsilon =2\frac{\epsilon'}{\varphi_{min}}$ (which transform the condition $\varepsilon' \leq 1$, into $\varepsilon \leq 2 \varphi_{min}$) and by choosing $\bar{C}=\frac{2C_2+C_3}{2}$ we obtain
\begin{equation*}
\begin{aligned}
\int_{\Omega_h^+} \tilde{v}_t^2 dx + & \int_{\Omega_h^+} \int_0^T  [(\dpiu_h \tilde{v})^2 + (\dmen_h \tilde{v})^2]  \le  \frac{1}{\varepsilon}\bar{C} +\varepsilon \|\dpiu_h f\|^2_{L_t^\infty L_x^2}.
\end{aligned}
\end{equation*}
\end{proof}

\begin{proposition}[Estimates for the linearized equation in $s$]\label{estimate_for_s}
Let $\tilde{s}$ be the solution of \eqref{eq:lin_PDE_sg_s}. Then for $T>0$ the following estimate on $\tilde{s}$ holds, for any $f$ satisfying \eqref{eq:hp_f}, we have
\begin{equation}\label{eq:estimation_linearzed_s}
    \sup_{t\in [0,T]} \|\tilde{s}(t,\cdot)\|_{L^2(\Omega_h^+)}^2 + \int_0^T\|\dpiu_h \tilde{s}(t,\cdot)\|^2_{L^2(\Omega_h^+)} dt\leq 
    \mu + \frac{1}{2} \|\dpiu_h f\|^2_{L_t^2 L_x^2},
\end{equation}
where $\mu$ is a measurable increasing function of $\eta, \lambda, C_0,\varphi_{min}, B,T,  \|\psi\|_{C^\beta_t}  $ in Assumptions \eqref{assumption:_bound_s0_c0_varphi} and in conditions \eqref{eq:hp_f}  but not on $K$ in \eqref{eq:hp_f}. In particular, by choosing $K$ in \eqref{eq:hp_f}  such that $K\geq 2\mu(\eta, \lambda, C_0,\varphi_{min}, B,T,  \|\psi\|_{C^\beta_t})$, we get
    \begin{equation}\label{eq:estimation_linearzed_s_2}
        \sup_{t\in [0,T]} \|\tilde{s}(t,\cdot)\|_{L^2(\Omega_h^+)}^2 + \int_0^T\|\dpiu_h \tilde{s}(t,\cdot)\|^2_{L^2(\Omega_h^+)} dt \le K.
    \end{equation}
\end{proposition}
\begin{proof}
Combining the estimates obtained for $u$ in \eqref{eq:continuity_u_Psi_H1} 
 and \eqref{eq:continuity_D_hu_Psi_H1}, and $\tilde{v}$ in \eqref{eq:stima_linearized_v} we can give a bound on the linearized solution $\tilde{s}=u+\tilde{v}$ only depending on the final time $T$:
\begin{align*}
  \sup_{t\in [0,T]} \int_{\Omega_h^+}\tilde{s}^2(t,x) dx& + \int_{\Omega_h^+} \int_0^T  [(\dpiu_h \tilde{s})^2 + (\dmen_h \tilde{s})^2] dx dt \le 2 \sup_{t\in [0,T]}\int_{\Omega_h^+} u(t,x)^2 dx \\
  &+ 2\int_{\Omega_h^+} \int_0^T  [(\dpiu_h u)^2 + (\dmen_h u)^2] dxdt \notag\\ 
   & + 2 \sup_{t\in [0,T]}\int_{\Omega_h^+} \tilde{v}(t,x)^2 dx + 2\int_{\Omega_h^+} \int_0^T  [(\dpiu_h \tilde{v})^2 + (\dmen_h \tilde{v})^2] dtdx \\ 
   & \le C \|\psi\|_{C^\beta_t}^2 + \frac{4}{\varepsilon}\bar{C} +2\varepsilon  \|\dpiu_h f\|^2_{L_t^\infty L_x^2}\label{eq:boundsf}
\end{align*}
If we choose $\varepsilon=\min\left(\frac{1}{4},2 \varphi_{min}\right)$ and $\mu=C \|\psi\|_{C^\beta_t}^2 + \frac{4}{\varepsilon}\bar{C}$. With this choice, we obtain inequality \eqref{eq:estimation_linearzed_s}. 
Finally, inequality \eqref{eq:estimation_linearzed_s_2} easily follows from \eqref{eq:estimation_linearzed_s} and $K\ge 2 \mu.$
\end{proof}
\begin{proposition}[Estimates for the nonlinear equation in $s$]\label{remark:fs}
Let $s$ be the  solution of \eqref{eq:system_s_c_discrete}. Then, for $T>0$ the following estimate on $s$ holds for some constant $K>0$ (big enough depending on the data of the problem but independent of $h$):

\begin{equation}\label{eq:estimation_linearzed_s_3} 
        \sup_{t\in [0,T]} \|s(t,\cdot)\|_{L^2(\Omega_h^+)}^2 + \int_0^T\|\dpiu_h s(t,\cdot)\|^2_{L_x^2} dt \le K
    \end{equation} \end{proposition}
\begin{proof}
    By Proposition \ref{estimate_for_s} one can easily derive that the same estimate is true for the nonlinear system, taking $f=s$ and applying  inequality \eqref{eq:estimation_linearzed_s_2}.
\end{proof}
\begin{remark}
    The bound \eqref{eq:estimation_linearzed_s_3}  on $s$ in Proposition \ref{remark:fs} , which is uniform in $0< h \leq 1$, guarantees that the solution $(s,c)$ to equation \eqref{eq:system_s_c_discrete} exists in $(L^2([0,T], H^1(\Omega_h^+))\cap L^{\infty}([0,T],L^2(\Omega_h^+)))^2$. 
\end{remark}

\begin{proposition}\label{proposition:ch}
Consider the solution $c$ to the second equation in the space-discrete system
\eqref{eq:system_s_c_discrete} with $ c_0 \in H^1(\Omega^+_h).$ The norm $\|c\|_{L^{\infty}([0,T],H^{1}(\Omega_h^+))}$  is uniformly bounded with respect to $0<h\leq 1$.
\end{proposition}
\begin{proof}
    Starting from the explicit definitions for $c$ in \eqref{eq:c_discrete} and for $g$ in \eqref{eq:explicit_g} and by imposing $f=s$,  we obtain that $g=c$. Thus, from \eqref{eq:dpiu_g_estimates} and the uniform bound for $s$ in \eqref{eq:estimation_linearzed_s_3} , we can provide the following bound for $c$:
    \begin{equation*}
        \sup_{[0,T]}\|\dpiu_h c\|^2_{L^2(\Omega_h^+)} \le c T \left(\|\dpiu_h c_0\|_{L^2(\Omega_h^+)}^2 + \kappa \|\dpiu_h s\|^2_{L_t^\infty L_x^2}\right)
    \end{equation*}
\end{proof}

 \subsection{Well-posedness of the nonlinear space-discrete model}
In this section we prove the existence and uniqueness of a weak solution to the semi-discrete system \eqref{eq:system_s_c_discrete}. 

\begin{theorem}\label{well-posedness_semi_discrete}
    Let us consider the system \eqref{eq:system_s_c_discrete} on $[0,T_{fin}]\times \Omega_h^+$. Suppose that Assumptions \ref{assumption:_bound_s0_c0_varphi} and condition  \eqref{eq:regularity_boundary_process} are satisfied.  
    For any $T\in [0,T_{fin}]$, there exists a pathwise unique weak solution $(s,c): [0,T]\times \Omega_h^+\rightarrow \mathbb R$.
\end{theorem}
\begin{proof}
Since for any fix $h\in \R_+, h>0$, all the Besov spaces on the lattice $\Omega_h^+$ are equivalent, with equivalence constants exploding as $h \rightarrow 0$, for obtaining the (local in time) well-posedness it is enough to check that the map
\begin{equation*}
    (s,c) \mapsto \frac{1}{2\varphi(c)}\left[\dpiu_h(\varphi(c)\dmen_h s)+\dmen_h(\varphi(c)\dpiu_h s)\right] -\lambda s c + B\lambda c s^2
\end{equation*}
is locally Lipschitz in the $L^2(\Omega_h^+)$ norm. But this is obvious, since it is a composition of smooth maps and, for any fixed $h>0$, the difference derivatives $\dpiu_h(\varphi(c)\dmen_h s)+\dmen_h(\varphi(c)\dpiu_h s)$ are smooth operators on $L^2(\Omega_h^+)$. The (local in time) well-posedness then follows from the Cauchy (Picard's) theorem till some  maximal time $T_h$ (depending on the initial condition). Combining this result with {the a priori estimates} obtained in Section \ref{section:apriori} (Proposition \ref{remark:fs} and Proposition \ref{proposition:ch} ), we can prove that $T_h=T_{fin}$, for any $T_{fin}\in \R_+$.
\end{proof}

\section{Convergence results}\label{convergence_result}

The main result of the paper regards the convergence of   the weak solution of the space-discrete system \eqref{eq:system_s_c_discrete} to the weak solution of the original system \eqref{eq:s2}-\eqref{eq:c2}, as the spatial discretization step $h$ tends to 0. Furthermore, we consider a fully discrete system, exploited in \cite{2024AMU_numerico}, and prove its convergence to the space-discrete. In such a way, we also gain the convergence of the numerical scheme to the solution of the continuum model, as space and time meshes shrink to zero. The main result is the first convergence result, i.e. the convergence  of the solution of the space-discrete system, which is shown by compactness arguments within the Besov spaces theory. 

\subsection{Interpolation and discretization}
Given, the space mesh size $h>0$, we denote by $(s_h,c_h):[0,T] \times \Omega_h^+ \rightarrow \mathbb R_+^2$ the solution of the space-discrete scheme \eqref{eq:system_s_c_discrete}, while $(s,c):[0,T] \times \Omega^+ \rightarrow \mathbb R_+^2$   is the solution of  system \eqref{eq:s2} - \eqref{eq:c2}. In the following we denote by $(s_{h,0},c_{h,0})$ and $(s_0,c_0)$ the initial function for the space-discrete and continuum system respectively. 

In order to extend   $s_h,c_h$ to $[0,T] \times \Omega$, a standard technique is to apply the  piecewise-linear interpolation. Since we have to consider the composition of  $(s_h,c_h)$ with some non-linear functions, it is more useful to adopt the \emph{piecewise-constant} extension.  Hence, we first define a piecewise constant \emph{extension operator}  able to extend the \textit{a priori} estimates of functions on $\Omega_h^+$ introduced in Section \ref{sec:space-discrete_setting}    to the continuum space $\Omega = \R_+$   \cite{devecchi_nicolay_2021elliptic,vladimir_montean,Ladyzenskaya}. We also recall some relevant properties of the piecewise-constant extension operator and other preliminary technical results. 

\begin{definition}[Piecewise-constant extension operator]\label{piece-contant-extension-operator}
    Let $w: \Omega_h^+ \to \R $ be a lattice function. Then, for every $x\in \Omega$, we define
    \begin{equation*}
        \mathcal{E}_h(w)(x):=\sum_{z\in \Omega_h^+} w(z) \mathbb{I}_{[z,z+h)}(x),
    \end{equation*}
    where $\mathbb{I}_K$ denotes the indicator function of the subset $K\subset \Omega$.
    We call the operator $\mathcal{E}_h$ the piecewise-constant  {extension} operator.
\end{definition}

\begin{theorem}[Extension operator properties]\label{theorem:extension}
Let $0 < s \leq 1$. The operator $\mathcal{E}_h$ is continuous from
$W^{s, p} (\Omega_h^+)$ to $B^{s \wedge \frac{1}{p}}_{p, \infty}
  (\Omega)$.
  
  \begin{itemize}
  \item[i)] For any $p \in [1, + \infty)$  we have that the operator norm is uniformly bounded in $0 < h \leq 1, $ i.e.
  \[ \sup_{0 < h \leq 1} \| \mathcal{E}_{h}
     \|_{\mathcal{L} (W^{s, p}(\Omega_h^+), B^{s \wedge \frac{1}{p}}_{p,
     \infty} (\mathbb{R}_+))} < + \infty,\]
     where $ \mathcal{L}$ denotes the Banach space of linear and bounded operators equipped with the natural norm.
  \item[ii)] In the case $s = 1$, $\mathcal{E}_{h}$ is continuous from $W^{1, p} (\Omega_h^+)$ into $B^{\frac{1}{p}}_{p,p} (\Omega)$ with the operator norm uniformly bounded for $0 < h \leq 1$.
  \smallskip
\item[iii)] 
  Let
  $ f_h $ be a sequence of functions in $ \, W^{s, p}
  (\Omega_h^+)$ such that  
  $$ \sup_{0 < h \leq 1} 
  \|  f_h \|_{W^{s, p} (\Omega_h^+)} < + \infty, \qquad \mathcal{E}_h (f_h) \rightarrow f\quad in \quad
  L^p (\Omega).$$
  Then $f \in B^s_{p, \infty}
  (\Omega)$. 
  \smallskip
  \item[iv)] 
  When $\, s = 1$ and $\, 1 < p < + \infty$ then $f \in
  W^{1, p} (\Omega)$, where $W^{1, p}(\Omega)$ is the Sobolev space of one time weakly differentiable functions with weak derivatives in $L^p (\Omega)$.
  \end{itemize}
\end{theorem}
\begin{proof}
    The interested reader may refer to \cite[Theorem 2.23 and Theorem 2.25]{devecchi_nicolay_2021elliptic}.
\end{proof}
 \begin{definition}
We denote by $\mathcal{D}_h$  a \emph{discretization operator} from the
space  $L^p(\Omega)$ into
$L^p(\Omega_h^+)$, with $p \in [1, + \infty]$, defined as
\begin{equation*}
  \mathcal{D}_h (\varphi) (z) :=  \frac{1}{h}
  \int_{z}^{z+h} \varphi (x) d x, \quad \varphi \in
  L^p(\Omega), \quad z \in \Omega_h^+.
\end{equation*}    
\end{definition}
\begin{lemma}\label{remark:DEh}
For $g\in L^p(\Omega)$ and $f\in L^q(\Omega_h^+)$, with $\frac{1}{p}+\frac{1}{q}=1$,  the following equality holds
\begin{equation}\label{eq:DEh} \int_{\Omega}g(x) \mathcal{E}_hf(x) dx= \int_{\Omega_h^+} \mathcal{D}_hg(z) f(z) dz.\end{equation}
\end{lemma}
\begin{proof}    The statement can be proved by a direct calculation.
\end{proof}
In the following, we prove some technical results that will be useful in establishing our main convergence theorem. 

\begin{definition}\label{remark:cinfty}
Let $g\in C^{\infty}(\Omega)$ be a smooth function. We denote by $g_h$ the function defined by applying the discretization operator $\mathcal{D}_h$ to $g$, namely
\begin{equation*}
g_h(z)=\mathcal{D}_h(g)(z),\quad z \in \Omega_h^+. 
\end{equation*}
\end{definition}
\begin{lemma}\label{lemma:cinfty}
Let $g\in C^{\infty}_0(\Omega)$ be a smooth function with compact support and let $g_h$ be the descretized function  as in Definition \ref{remark:cinfty}. Then, for any $p \in [1, +\infty]$, $s < \frac{1}{p}$, we have 
\[
\lim_{h \rightarrow 0} \|\mathcal{E}_h(g_h) - g \|_{B^s_{p,p}(\Omega)} =0,
\]
where $\mathcal{E}_h$ is the extension operator introduced in Definition \ref{piece-contant-extension-operator}.\\
Furthermore, we get 
\[ \lim_{h \rightarrow 0} \|\mathcal{E}_h(D^+_h g_h) - \partial_xg \|_{B^s_{p,p}(\Omega)} = \lim_{h \rightarrow 0} \|\mathcal{E}_h(D^-_hg_h) - \partial_xg \|_{B^s_{p,p}(\Omega)}= 0\]
and
\[\lim_{h \rightarrow 0} \|\mathcal{E}_h(\Delta^D_h g_h) - \partial_x^2g \|_{B^s_{p,p}(\Omega)}=0\]
\end{lemma}
\begin{proof}
 Since $g \in C^{\infty}_0(\Omega)$, i.e. is a compactly supported smooth function, we obtain $g_h(z)=\mathcal{D}_h(g)(z)=g(z)+o(z)$, $\dpiu_hg_h(z)=\frac{1}{h}\int_0^h{\dpiu_hg(z+s)ds}=\mathcal{D}_h(\partial_x(g))(z)+o(h)=\partial_x(g)(z)+o(h)$, and similarly for the other terms. \\
From these relations, it is very simple to prove that $\mathcal{E}_h(g_h)$, $\mathcal{E}_h(\dmen_hg_h)$, $\mathcal{E}_h(\dpiu_hg_h)$ and $\mathcal{E}_h(\Delta^D g_h)$ converge weakly, as distributions in $\mathcal{S}'(\Omega)$, to $g$, $\partial_x(g)$, $\Delta_x g$, respectively.\\
Furthermore, by Theorem \ref{theorem:extension}, we get
\[\sup_{0<h\leq 1}  \left\{\|\mathcal{E}_h(g_h)\|_{B^{1/p}_{p,p}},\|\mathcal{E}_h(\dpiu_hg_h)\|_{B^{1/p}_{p,p}}, \|\mathcal{E}_h(\dmen_hg_h)\|_{B^{1/p}_{p,p}}, \|\mathcal{E}_h(\Delta_h g_h)\|_{B^{1/p}_{p,p}}\right\} < +\infty, \]
and thus, for any $s<\frac{1}{p}$, using the compactness of the immersion of $B^{1/p}_{p,p}(\Omega)$ into $B^s_{p,p}(\Omega)$, the result follows.
\end{proof}

The previous lemma allows us to establish the following technical result concerning the convergence of the products of converging sequences.

\begin{proposition}\label{proposition:product}
Consider an $f \in H^1(\Omega)$ such that $f_h \in H^1(\Omega_h^+)$ and, as $h \downarrow 0$, 
$\mathcal{E}_h(f_h) \rightarrow f,$ 
where the convergence is in $L^2(\Omega)$.  Then $\mathcal{E}_h(\dpiu_h f_h)$ and $\mathcal{E}_h(\dmen_h f_h)$ converge  to $\partial_x f$ weakly in $L^2(\Omega)$, where $\partial_x f $ is understood  to be the weak derivative of the function $f\in H^1(\Omega)$.\\
Furthermore, for any other $k_h \in H^1(\Omega_h^+)$, such that $\mathcal{E}_h(k_h)$ converges to $k \in H^1(\Omega)$, then, for any $p\in [1,+\infty)$, $\mathcal{E}_h(k_h f_h)$ converges  to $k f$, strongly in $L^p(\Omega)$.
\end{proposition}
\begin{proof}
Let $g \in C^{\infty}_0(\Omega)$ be a smooth function with compact support. Denoting by $g_h=\mathcal{D}_h(g)$ its discretization according with Definition \ref{remark:cinfty}, by Lemma \ref{remark:DEh} and the discrete integration by parts formula, we note that 
\[\int_{\Omega}g(x)\mathcal{E}_h(D_h^+f)(x)dx= \int_{\Omega_h^+} g_h(z) D_h^+f_h(z) d z=- \int_{\Omega_h^+} D_h^-g_h(z) f_h(z) d z.\]
Using the fact that $D_h^-g_h$ converges to $\partial_x g$ in $L^2(\Omega^+_h)$, and $f_h$ converges weakly to $f$ in $L^2(\Omega^+_h)$ the convergence of $\mathcal{E}_h(D_h^+ f_h)$ to $\partial_x f$ in $\mathcal{S}'(\Omega)$ is provided. 
Since $\|f_h\|_{H^1(\Omega_h^+)}$ is uniformly bounded, we obtain that $\|\mathcal{E}_h(D_h^+f)\|_{L^2(\Omega)}$ is also uniformly bounded and thus $\mathcal{E}_h(D_h^+f_h)$ converges to $\partial_xf$, weakly in $L^2(\Omega)$.
\smallskip

The second part of the proposition follows from the fact that $\mathcal{E}_h(k_h f_h)=\mathcal{E}_h(k_h)\mathcal{E}_h(f_h)$, the previous results on the boundedness of $\mathcal{E}_h$ and on the properties of compact embeddings of Besov spaces.
\end{proof}

\subsection{Convergence of the space-discrete to the continuum system }
Let now refer to the continuum model \eqref{eq:deterministic_intro} with   the initial and boundary conditions given in \eqref{eq:stochastic_boundary_condition} and \eqref{eq:discretesystem_Initial_boundary}. Recall that by this is equivalent to  \eqref{eq:s2} and \eqref{eq:c2}.
   
\begin{definition}\label{def:continuous_weak_solution}
    We say that the couple $( s, c)$ is a weak solution for the system \eqref{eq:deterministic_intro} if for any $(\phi_1,\phi_2) \in C_0^\infty([0,T] \times \Omega) \times C_0^\infty([0,T] \times \Omega)$
    
    \begin{eqnarray}
        \int_0^T\int_\Omega  \varphi s\partial_t \phi_1 dt dx &=& \int_0^T\int_\Omega \varphi(c) \nabla s \nabla \phi_1 dt dx +\int_0^T\int_\Omega \lambda\varphi(c) s c\, \phi_1 dt dx;\label{eq:weak1} \\
        \int_0^T\int_\Omega  c\, \partial_t\phi_2 dt dx& =& \int_0^T\int_\Omega \lambda\varphi(c) s c \,\phi_2 dtdx. \label{eq:weak2}
    \end{eqnarray}
    
\end{definition}

We report that the author in \cite{1977_ball} has proved that, under the hypothesis that the operator $\partial_{xx}$ in \eqref{eq:s2}  generates a strongly continuous semigroup of bounded linear continuous operators and the term $b_c\partial_x s+\tilde{\gamma_c}s$ is $L^1$ bounded,  then the weak solution is a mild solution  to \eqref{eq:s2}. We can directly apply the cited result, but we still need to ensure that the \textit{weak formulation} as given in Definition \ref{def:continuous_weak_solution} is equivalent to the weak formulation for \eqref{eq:s2}. In doing this, we should consider an opportune test function $\phi_1$.

\begin{lemma}\label{weak_versus_mild}
The   weak solution $(s,c)$   \eqref{eq:weak1}-\eqref{eq:weak2}, with \eqref{eq:stochastic_boundary_condition} and \eqref{eq:discretesystem_Initial_boundary}, is a mild solution for the system \eqref{eq:deterministic_intro}, or equivalently of  \eqref{eq:s2}-\eqref{eq:c2},  in the sense of Definition \ref{def:mild_solution_s_nonlin}.
\end{lemma}
\begin{proof}
Let $(s,c)$ be the mild solution to the system \eqref{eq:s2}-\eqref{eq:c2}  with initial and boundary condition \eqref{eq:stochastic_boundary_condition} and \eqref{eq:discretesystem_Initial_boundary}. By the result in \cite{1977_ball}, the mild solution to equations \eqref{eq:s2} and \eqref{eq:c2} is the same as the \emph{weak solution} to the equations \eqref{eq:s2} and \eqref{eq:c2}, i.e. $(s,c)$ is a mild solution to equations \eqref{eq:s2} and \eqref{eq:c2} if and only if for any $g_0 \in C^{\infty}_0([0,T]\times \Omega)$, namely $g_0$ is a smooth function with compact support, we have 
\begin{eqnarray} \label{weak-s}
   -\int_0^T\int_{\Omega} \partial_tg_0(t,x)    s(t,x) dt dx   &=&\int_0^T\int_{\Omega}	\partial_x^2g_0(t,x) s(t,x) dtdx\label{eq:weak22}
   \\
   &&+\int_0^T\int_{\Omega} g_0(t,x) (b_c(t,x) \partial _x s+ \gamma_c(t,x) s(Bs- 1)) dt dx  \nonumber
\end{eqnarray}	
where $b_c$ and $\gamma_c$ are as in  \eqref{eq:b_c_gamma_c} and $c$ is given by the explicit formula \eqref{eq:df_mild_solution_c}. 
So, what remains to prove is that the weak formulation in Definition \ref{def:continuous_weak_solution} implies the weak formulation given by equation \eqref{eq:weak22}.
In order to prove this statement, we consider $g_0 \in C^{\infty}_0([0,T]\times \Omega)$ and a smooth mollifier $a_\varepsilon\in C^{\infty}_0(\R^2)$ with compact support, and we choose the following test function $$\phi_1(t,x) = g_0(t,x)\left(a_\varepsilon^{t,x} \ast \frac{1}{\varphi}\right).$$  By integrating by parts the first integral in \eqref{eq:weak1} we get:
\begin{align*}
    \int_0^T \int_{\Omega} \partial_t (\varphi(c)s) &  g_0(t,x)\left(a_\varepsilon^{t,x} \ast \frac{1}{\varphi}\right) = -\int_0^T \int_{\Omega} \varphi(c)s \partial_t \left( g_0(t,x) \left(a_\varepsilon^{t,x} \ast \frac{1}{\varphi}\right) \right) =\\
     & = - \int_0^T \int_\Omega \varphi(c)\left(a_\varepsilon^{t,x} \ast \frac1\varphi\right) s \partial_t g_0  - \int_0^T \int_\Omega \varphi(c)s g_0 \left(a_\varepsilon^{t,x} \ast   \frac{B \lambda s c }{\varphi}\right)
\end{align*}
Taking the limit for $\varepsilon\to 0$ we have
\begin{equation}\label{eq:limite_weak_t}
    \int_0^T\int_\Omega \partial_t(\varphi(c)s)g_0(t,x)\left(a_\varepsilon^{t,x}\ast \frac1\varphi\right) \longrightarrow -\int_0^T\int_\Omega s\partial_t g_0 -\int_0^T\int_\Omega B\lambda s^2 c g_0.
\end{equation}
 By applying the product rule to the second term in \eqref{eq:weak1}, we obtain:
\begin{equation*}
\begin{aligned}
    \int_0^T\int_\Omega \varphi(c)\partial_{x}s \partial_x & \left(g_0\left(a_\varepsilon\ast \frac1\varphi\right)\right)  = \\
     & =\int_0^T\int_\Omega \varphi(c)\left(a_\varepsilon\ast \frac1\varphi\right) \partial_x s \partial_x g_0 -\int_0^T\int_\Omega \varphi(c)  g_0 \left(a_\varepsilon\ast \frac{b_c(t,x)}{\varphi}\right)\partial_x s,
\end{aligned}
\end{equation*}
which gives, as $\varepsilon\to 0$:
\begin{equation}\label{eq:limite_weak_x}
    \int_0^T\int_\Omega\varphi(c)\partial_x s \partial_x\left(g_0\left(a_\varepsilon^{t,x}\ast\frac1\varphi\right)\right) \longrightarrow \int_0^T\int_\Omega \partial_x s \partial_x g_0 -\int_0^T\int_\Omega \partial_x s b_c g_0.
\end{equation}
Inserting the limits \eqref{eq:limite_weak_t} and \eqref{eq:limite_weak_x} in \eqref{eq:weak1} we get
\begin{equation*}
  - \int_0^T \int_\Omega s \partial_t g_0 -\int_0^T \int_\Omega  B\lambda s^2 c g_0 + \int_0^T \int_\Omega \partial_x s \partial_x g_0 - \int_0^T \int_\Omega \partial_x s b_c g_0 = -\int_0^T \int_\Omega \lambda s c g_0
\end{equation*}
 Integrating by parts the first term on the left-hand side of the previous equality we have:
\begin{equation*}
\int_0^T \int_\Omega (\partial_t s) g_0 = \int_0^T \int_\Omega (\partial_{xx}s)g_0 + \int_0^T \int_\Omega \left(\partial_x s b_c + B\lambda s^2 c - \lambda s c\right) g_0,
\end{equation*}
which is exactly the weak formulation for equation \eqref{weak-s} against a test function $g_0 \in C_0^\infty(\Omega)$.
\end{proof}
The following is our main convergence result. 
 
\begin{theorem} [Convergence theorem]
    \label{def:weak_formulation}
    Let $(s_h,c_h)$ be the unique solution to the space-discrete system defined in $[0,T]\times \Omega_h^+$ as
    \begin{equation}\label{eq:semi_discrete_sys}
        \begin{split}
            \partial_t (\varphi(c_h) s_h) & = \frac12 \left( \dpiu_h(\varphi(c_h) \dmen_h s_h)+\dmen_h (\varphi(c_h) \dpiu_h s_h)\right) -\lambda \varphi(c_h)s_hc_h, \\
            \partial_t c_h & = - \lambda \varphi(c_h)  s_h c_h,
        \end{split}
    \end{equation}
    with initial and boundary conditions given by \eqref{eq:stochastic_boundary_condition} and \eqref{eq:discretesystem_Initial_boundary}.
    Let us assume that the  initial conditions for the space-discrete system $( s_{h,0}(x),  c_{h,0}(x))$ converge to $(s_0(x),c_0(x))$, i.e. $( s_{h,0}(x),  c_{h,0}(x))$ is uniformly bounded in $H^1(\Omega_h^+)$, and $(\mathcal{E}_h( s_{h,0}(x)), (\mathcal{E}_h( c_{h,0}(x)))$ converge to $(s_0(x),c_0(x))$ in $L^2(\Omega_+)$. Introducing  $\bar{s}_h=\mathcal{E}_h(s_h)$ and $\bar{c}_h=\mathcal{E}_h(c_h)$, we have that, for any $p\in[1,2]$ and for any $k < \frac{1}{p}$,  $(\bar{s}_h,\bar{c}_h)$ converges in $L^p([0,T],B^k_{p,p}(\Omega))$ to the unique solution $(s,c)$ to the system \eqref{eq:weak1}-\eqref{eq:weak2}.
\end{theorem}

\begin{proof} 
By the a priori estimates proved in Proposition \ref{remark:fs}, we have that 
\[\sup_{0<h\leq 1}\|s_h \|_{L^{2}([0,T],H^1(\Omega_h^+))}  \]
which implies, up to a suitable subsequence, that $\bar{s}_h=\mathcal{E}_h(s_h)$, $D^{\pm}_h\bar{s}=\mathcal{E}_h(D^{\pm}s_h)$ converge to some function $\hat{s}$ and to its derivatives $\partial_x\hat{s}$, respectively, weakly in $L^2([0,T]\times \Omega_h^+).$ \\

Consider ${k}\in C^{\infty}_0([0,T] \times \Omega)$, and let $k_h=\mathcal{D}_h(k)$ be the discretization of $k$ in the space variable. Integrating by parts in the weak formulation of \eqref{eq:semi_discrete_sys} against a test function $k_h$, we get
\begin{equation*}
\begin{aligned}
   \int_0^T\int_{\Omega_h^+}\partial_t(k_h) (\varphi(c_h) s_h)dzdt & = \frac12 \int_0^T\int_{\Omega_h^+}\left( \dmen_h(k_h) \varphi(c_h) \dmen_h s_h+\dpiu_h(k_h)\varphi(c_h) \dpiu_h s_h\right) dz dt\\
   & \quad +\lambda \int_0^T\int_{\Omega_h^+} \varphi(c_h)s_hc_h k_h dzdt\\
   \int_0^T\int_{\Omega_h^+} \partial_t(k_h) c_h dzdt & = +\lambda \int_0^T\int_{\Omega_h}  \varphi(c_h)  s_h c_h k_h dz dt   
\end{aligned}
\end{equation*}
By Lemma \ref{remark:DEh} we get that
\begin{equation*} 
\begin{aligned}
    \int_0^T\int_{\Omega}\partial_t(k) (\varphi(\bar{c}_h) \bar{s}_h)dxdt & = \frac12 \int_0^T\int_{\Omega}\left( \bar{k}^-_h \varphi(\bar{c}_h) \bar{s}^-_h+\bar{k}^+_h\varphi(\bar{c}_h)  \bar{s}^+_h\right) dx dt\\
    &\quad +\lambda \int_0^T\int_{\Omega}k \varphi(\bar{c}_h)\bar{c}_h\bar{s}_h dxdt,\\
    \int_0^T\int_{\Omega} \partial_t(k) \bar{c}_h dxdt & =+ \lambda \int_0^T\int_{\Omega} k \varphi(\bar{c}_h) \bar{c}_h \bar{s}_h  dx dt,  
\end{aligned} 
\end{equation*}
where $\bar{c}_h=\mathcal{E}_h(c_h)$, 
$\bar{s}_h=\mathcal{E}_h(s_h)$,
$\bar{k}^{\pm}_h=\mathcal{E}_h(D_h^{\pm}k_h)$ and $\bar{s}^{\pm}_h=\mathcal{E}_h(D_h^{\pm}s_h)$ .

We need $\bar{c}_h$ to converge at least strongly in $L^2([0,T]\times\Omega)$ to $\hat{c}$. We also need that $\bar{k}^{\pm}$ converges to $k$, strongly in $L^{\infty}([0,T]\times\Omega)$ (but this holds by Lemma \ref{lemma:cinfty}). Finally, we need that $\varphi(\bar{c}_h)\bar{c}_h$ converges to $\varphi(\hat{c})\hat{c}$ strongly in $L^2([0,T]\times\Omega)$.
\\ 
From Proposition \ref{proposition:ch}, we have that 
\[\sup_{0<h\leq 1} \|c_h \|_{L^{\infty}([0,T],H^{1}(\Omega_h^+))} < +\infty\]
when $c_0\in H^1(\Omega)$. We can exploit the second equation of the system \eqref{eq:semi_discrete_sys} to deduce that
\[
\sup_{0 < h \leq 1} \|c_h \|_{H^1([0,T],L^p(\Omega_h^+))}<+\infty
\]
Using interpolation inequalities on the previous two quantities we get that for any $q \in [1,+\infty)$ there is $\alpha,\delta>0$ small enough such that
\[ \sup_{0<h\leq 1}\|c_h\|_{B^{\alpha}_{\infty,\infty}([0,T],B^{\delta}_{q,q}(\Omega_h^+))} < +\infty\]
By Theorem \ref{theorem:extension}, this implies that 
\[ \sup_{0<h\leq 1}\|\varphi(\bar{c}_h) \bar{c}_h\|_{B^{\alpha}_{\infty,\infty}([0,T],B^{\delta}_{q/2,q/2}(\Omega))} < +\infty\]
and, by compact immersion of Besov spaces and Proposition \ref{proposition:product}, we get that $\varphi(\bar{c}_h)\bar{c}_h$ converges to $\varphi(\hat{c})\hat{c}$ strongly in  $L^2([0,T],L^2(\Omega))=L^2([0,T] \times \Omega)$. In a similar way it is possible to prove that $\bar{c}_h$ converges to $\hat{c}$ in $L^2([0,T]\times \Omega)$.\\
Putting all the previous convergences together we obtain, taking $h\rightarrow 0$, that 
\begin{equation*}
\begin{aligned}
   \int_0^T\int_{\Omega}\partial_t(k) (\varphi(\hat{c}) \hat{s})dxdt & = \frac12 \int_0^T\int_{\Omega}\left( \partial_x(k) \varphi(\hat{c}) \partial_x(\hat{s})+\partial_x(k)\varphi(\hat{c})  \partial_x(\hat{s})\right) dx dt\\
   & \quad +\lambda \int_0^T\int_{\Omega}k \varphi(\hat{c})\hat{c}\hat{s} dxdt \\
  \int_0^T\int_{\Omega} \partial_t(k) \hat{c} dxdt & = + \lambda \int_0^T\int_{\Omega} k \varphi(\hat{c}) \hat{c} \hat{s}  dx dt. 
  \end{aligned}
\end{equation*}
Thus, by the uniqueness of the weak solution $ (s,c)$ to equation \eqref{eq:s2}-\eqref{eq:c2}, we get that $(\hat{s},\hat{c})=(s,c)$. Since $\bar{s}_h=\mathcal{E}_h(s_h)$ and $\bar{c}_h=\mathcal{E}_h(c_h)$ converge to $\hat{s}=s$ and $\hat{c}=c$ respectively, we have that the solution to the semi-discrete equation converges to the solution to the corresponding equation in the continuum space $[0,T]\times \Omega$.
\end{proof}

\subsection{A fully discrete system: convergence results}
Let us introduce a fully discrete numerical scheme for the system  \eqref{eq:deterministic_intro}. See also \cite{2024AMU_numerico} for a similar scheme in the case of a specific dynamical boundary process. Let us discretize in time  the space-discrete scheme \eqref{eq:semi_discrete_sys}    by forward finite differences. The convergence of the fully discrete system to the space-discrete one is provided here to guarantee the desired convergence of the numerical scheme to the unique mild solution of the original problem.

\smallskip

Let $k,h$ be the time and space dimensions of an equi-partition of $[0,T]\times \mathbb R_+$  $$  \Xi_k \times \Omega_h^+ =[0, k  , 2k , ..., N_k =T] \times \Omega_h^+,$$ with  $ k=\frac{T}{N}$ and $\Omega_h^+=h\mathbb Z_+$. Furthermore,  let  us put ${ {\bar \Delta}} = \frac{k}{h^2}$. 

\begin{notations}[The discrete functions]\label{not:discrete_functions}
    The following notations are considered: for any $(n,m)\in \{0,\ldots,N\}\times\mathbb{N}$, $(t_n,x_m)=(nk,mh)$,  so that for a function $g$, $g_m^n$ denotes the discrete counter part, i.e.  $g_m^n=g(t_n,x_m)=g(nk,mh)$. If we would like to make explicit the dependence upon $h$ and $k$, we denote by $(s_h,c_h)$ the solution to the space-discrete equation \eqref{eq:semi_discrete_sys} and by $(s_h^{k},c_h^{k})$ the solution to the fully discrete, both in space and in time, numerical scheme. However when it is not strictly needed, in particular when $k$ and $h$ are fixed, we omit them; so that  $(s_{h,m}^{k,n},c_{h,m}^{k,n})=(s_{h}^{k}(t_n,x_m),c_{h}^{k}(t_n,x_m))=:(s_{m}^{n},c_{m}^{n})$ and  $(s_{h,m},c_{h,m})=(s_{h}(x_m),c_{h}(x_m))=:(s_{m},c_{m})$. 
\end{notations}
\begin{notations}(The interpolated functions) \label{not:interpolation_functions}
    Let us denote by { $\bar{s}^{k}_{h}(t,x)$} the function $\bar{s}^{k}_{h}:[0,T] \times \R_+ \rightarrow \R$ such as if $t=t_n, x=x_m$,  $\bar{s}^{k}_{h}(t_n,x_m):=s^{k,n}_{h,m}$,  the solution to the numerical scheme; otherwise if $t\not= t_n$ and $x\not= x_m$, $\bar{s}^{k}_{h}(t,x)$  is extended by linear interpolation in the $t$ variable (for any fixed $x$) and pice-wise constant interpolation in $x$.\\ Furthermore,  let us denote by $\psi^k(t):[0,T] \rightarrow \R_+$ the boundary process such that $\psi^{k,n}:=\psi^k(t_n)$ so that $\psi^{k,n} $ is a converging discrete approximation of  $\psi$, while  for $t\not=t_n$  $\psi^k(t)$ is extended by linear interpolation in time.
\end{notations}

In order to get a (fully discrete) numerical  scheme, let us first observe that, by \eqref{eq:chain_rule_D},
\begin{equation*}
\begin{aligned}
    & \frac12 \left[ \dpiu_h  (\varphi(c_m^n)\dmen_hs_m^n)+\dmen_h (\varphi(c_m^n)\dpiu_h s_m^n)\right]  =   \varphi(c_m^n)\frac{s_{m+1}^n-2s_m^n+s_{m-1}^n}{h^2}   \\
    &\qquad \quad  + \frac{1}{2h^2}\big[s_{m+1}^n(\varphi(c_{m+1}^n)-\varphi(c_m^n)) + s_m^n(-\varphi(c_{m+1}^n)-\varphi(c_{m-1}^n)+2\varphi(c_m^n)) \\
    &\qquad \quad +s_{m-1}^n(-\varphi(c_m^n)+\varphi(c_{m-1}^n)) \big].
\end{aligned}
\end{equation*}

Hence, by  applying the finite difference operator to the time derivative, from \eqref{eq:discretesystem} and \eqref{eq:system_s_c_discrete} we obtain the following scheme for (s,c): for any $(n,m)\in \{\textcolor{red}{1},\ldots, N\}\times\mathbb{N}$
   \begin{eqnarray*}\label{scheme_s}
   s_m^{n+1} & =& \bar{\Delta}\left(1+\frac{\varphi(c_{m+1}^n)-\varphi(c_m^n)}{2\varphi(c_m^n)}\right)s_{m+1}^n +  \bar{\Delta}\left(1-\frac{\varphi(c_{m}^n)-\varphi(c_{m-1}^n)}{2\varphi(c_m^n)}\right)s_{m-1}^n\\
   && + \left(1 -2\bar{\Delta} - \frac{\varphi(c^n_{m+1})+\varphi(c^n_{m-1})}{2\varphi(c^n_m)}\bar{\Delta}-\lambda k c_m^n (1-B s_m^n)\right)s_m^n;
\\
c_m^{n+1} &=& c_m^n \,e^{-  \lambda k s_m^n\, \varphi(c_m^n)};
\end{eqnarray*}
 coupled with the boundary conditions $s_0^{n+1}= \tilde{\psi}^{n+1}; \quad c_0^{n+1}=c_0^n \,e^{-  \lambda k \tilde{\psi}^{n}\, \varphi(c_0^n)}.$ Note that $\tilde{\psi}^n=\tilde{\psi}(t_n)$ is a converging approximation in time of the stochastic boundary condition.
\begin{remark}
Note that in   \cite{2024AMU_numerico} a slightly different scheme has been introduced in a bounded space domain, so that a Neumann boundary condition has been employed.   The only difference in the derivation of the scheme is that  in \cite{2024AMU_numerico}  a centered finite differences are considered. Furthermore, a specific random boundary condition  given by \eqref{eq:Pearson_SDE_boundary} is used.
\end{remark} 

\subsubsection{Convergence of the fully discrete to the space-discrete system}

\begin{proposition}[Scheme stability]\label{boundness_stability}
Let us  suppose that 
\begin{eqnarray} \label{scheme_stability_1}
    k &\le& \frac{h^2}{2} \wedge \frac{h^2}{2+\lambda c_0h^2 (1-B{\eta})}, \quad   B<0;\\ \label{scheme_stability_2}
    k &\le&  \frac{h^2}{2} \wedge \frac{h^2}{2+\frac{\varphi_{max}}{\varphi_{min}}+\lambda c_0 h^2}, \quad   0<B<1/\eta.
\end{eqnarray}
Then the solution of the system is such that, for any $(n,m)\in \{0,1,\ldots,N\}\times \mathbb{N}$ and for any trajectory $\omega\in\Omega$
\begin{equation*}
    (s^n_m(\omega), c_m^n(\omega)) \in \mathbb [0,{\eta}) \times   [0,{c}_0].
\end{equation*}
\end{proposition} 
\begin{proof} It is easy to prove that  $c^{n+1}_m \in [0,c_0], $ {whenever $s_m^n \in [0,\eta)$}  ( \cite{2024AMU_numerico}).
 We show that  the first coefficient in the r.h.s of the scheme for $s$ is positive; indeed, \begin{eqnarray*}
   1+\frac{\varphi{(c^n_{m+1})}-\varphi{(c^n_{m})}}{2\varphi(c_m^n)}
     &=&   \frac{2\varphi(c_m^n)+\varphi{(c^n_{m+1})}-\varphi{(c^n_{m})}}{2\varphi(c_m^n)} =
    \frac{\varphi{(c^n_{m+1})}+\varphi{(c^n_{m})}}{2\varphi(c_m^n)} >0.
\end{eqnarray*}
Analogously,
 \begin{eqnarray*}
   1-\frac{\varphi{(c^n_{m})}-\varphi{(c^n_{m-1})}}{2\varphi(c_m^n)}
     &=&   \frac{2\varphi(c_m^n)-\varphi{(c^n_{m})}+\varphi{(c^n_{m-1})}}{2\varphi(c_m^n)} 
     = \frac{\varphi{(c^n_{m})}+\varphi{(c^n_{m-1})}}{2\varphi(c_m^n)} >0.
\end{eqnarray*}

An upper bound for $s^{n+1}_m$ can be derived as in \cite{2024AMU_numerico} for $B<0$ and for { $0<B<1/\eta$.} 
For the proof of the lower bound, we need to prove that also the last coefficient of the discrete equation for $s$ is positive; the case in which $B<0$ has been discussed in \cite{2024AMU_numerico} while for $B>0$ we have the following 
\begin{eqnarray*}
    & 1- \bar{\Delta} \left(2  -\frac{\varphi(c^n_{m+1})+\varphi(c^n_{m-1})}{2\varphi(c^n_m)}\right) - \lambda k    c_m^n \left(1-B s_m\right)\ge 1-2 \frac{k}{h^2}-\frac{\varphi(c^n_{m+1})+\varphi(c^n_{m-1})}{2\varphi(c^n_m)}\frac{k}{h^2}-  \lambda k   \bar{c}_0 \\
     &\ge1-2 \frac{k}{h^2}-\frac{2\varphi_{ax}}{\varphi_{min}}\frac{k}{h^2}-  \lambda k    \bar{c}_0  =\frac{1}{h^2}\left[h^2-k\left(2  + \frac{2\varphi_{max}}{\varphi_{min}}+ h^2\lambda     \bar{c}_0  \right)\right]
\end{eqnarray*}
Then stability  thesis is achieved. 
\end{proof}

In order to gain the convergence of the fully discrete system to the space-discrete one we plan to apply  the Lax equivalence theorem. Since the stability has been established, it is sufficient to ensure the consistency of the numerical scheme. 
\begin{proposition}
Fix $h\in \R_+$. Let us suppose that the stability  conditions of Proposition \ref{boundness_stability} on $k$ holds, and that {$\psi^k \rightarrow \psi$} in $\mathcal{C}^\beta([0,T]), \forall \beta \in (1/4,1/2). $ Then 
\begin{equation*}
\lim_{k \rightarrow 0} \sup_{t_n \in [0,T]} \| s_{h,m}^{k,n}-s_h(t_n,\cdot) \|_{H^1(\Omega_h^+)} = 0.
\end{equation*}
\end{proposition}

\begin{proof}
 
 Let us estimate the  local truncation error (LTE) in time,  defined by replacing the solution $s_m^n$ by the \textit{true} solution $s_m(t_n)$ . In fact, the solution $s_m(t_n)$ does not exactly satisfy the equation and the discrepancy is accounted for by the LTE. We denote it by
\begin{equation}\label{eq:lte}
    \tau_m^n : = \frac{s_m(t_n+k)-s_m(t_n)}{k} - \widetilde{f}_m(t_n)
\end{equation}
where\begin{align*}
    \widetilde{f}_m(t_n) : &= \frac{1}{2\varphi(c_m(t_n))}\left[\dpiu_h\varphi(c_m(t_n))\dmen_h s_m(t_n)+\dmen_h\varphi(c_m(t_n))\dpiu_h s_m(t_n)\right]\\
    &-\lambda c_m(t_n)s_m(t_n)+\lambda B c_m(t_n)s_m^2(t_n)
\end{align*}
Since we do not have an explicit representation for $s_m(t_n)$, we should provide a Taylor expansion in the nodes of the time  as in the following
\begin{equation*}
    s_m(t_n+    k):= s_m(t_n)+ k\partial_t s_m(t_n) + \frac{k^2}{2}\partial_t^2 s_m(t_n)+ \mathcal{O}(k^3)
\end{equation*}
and by replacing the previous expression in \eqref{eq:lte} we get
\begin{align*}
    \tau_m^n : & = \frac{s_m(t_n)}{k} + \partial_t s_m(t_n) + \mathcal{O}(k) - \frac{s_m(t_n)}{k} - \widetilde{f}_m(t_n)\\
    & = \partial_t s_m(t_n) - \widetilde{f}_m(t_n) + \mathcal{O}(k) = \mathcal{O}(k),
\end{align*}
since $s_m(t_n)$ is a solution of \eqref{eq:system_s_c_discrete}.
Therefore, the method is (locally) consistent with order 1 in time. To prove that the global error is $O(k)$ it is sufficient to have $\widetilde{f}_m$ locally Lipschitz. And this is true as we discussed in the proof of Theorem \ref{well-posedness_semi_discrete}. Indeed, by  the a priori estimates on $(s_h,c_h)$ proved in Section 6, the local Lipschitzianity property of $\widetilde{f}_m(t_n)$ is equivalent to the global Lipschitzianity property (see, e.g., \cite{quarteroni2010numerical}, Section 11.3.).
\end{proof}

\begin{remark}
As a consequence of the proof it is possible to see, that, for fixed $h \in \R_+$, the convergence rate in $k$ is linear. Unfortunately, since the equivalence constants between the spaces $H^1(\Omega_h^+)$ and $L^2(\Omega_h^+)$ and the Lipschitz constant in the function $\widetilde{f}_m$ explode as $h \rightarrow 0$, the rate of convergence is not uniform in $h$. 
\end{remark}
 
\subsubsection{Convergence of the fully discrete to the continuum system}
  
Let us consider the Notation \ref{not:interpolation_functions}. We finally provide the last convergence result.

\begin{theorem}[Final convergence result]
 Le us suppose that $( s_{h,0}(x),  c_{h,0}(x))$ converge to $( s_{0}(x), c_{0}(x))$, in the sense of Theorem  \ref{def:weak_formulation}, and that  the boundary functions $\psi^k(t)$,  as $k$ tends to zero, converges to $\psi(t)$ in $L^{\infty}([0,T]).$ Then
    there exists a function $k(h)$, growing with respect to $h$, and such that $k(h)\rightarrow 0$ as $h \rightarrow 0$, for which,
\begin{equation*}
\lim_{h \rightarrow 0}\|s_h^{k(h)}-s \|_{L^{\infty}([0,T],H^1(\R_+))}=0
\end{equation*}
\end{theorem}
\begin{proof}

Fix $h\in \R_+$. Since the fully discrete scheme, under the stability conditions of Proposition \ref{boundness_stability}, converges to the space-discrete scheme, then there is $k(h)$, satisfying the hypotheses of Proposition \ref{boundness_stability}, such that 
\begin{equation*}
    \sup_{t_n\in \R_+}\|s_{h,\cdot}(t_n)_{h,\cdot}^{k(h),n}-s_h\|_{H^1(\R_+)} \leq h. 
\end{equation*}
We can choose the function $k(h)$ to be increasing and such that $k(h)\rightarrow 0$ as $h \rightarrow 0$.  
With this in mind, by Theorem \ref{def:weak_formulation}, we get 
\begin{align*} 
\lim_{h \rightarrow +\infty}\|s-s_{h}^{k(h)}\|_{L^{\infty}([0,T],H^1(\R_+))} \le& \lim_{h \rightarrow +\infty} \|s-s_h\|_{L^{\infty}([0,T],H^1(\R_+))}\quad  + \\ 
&\lim_{h \rightarrow +\infty} \|s_h-s_h^{k(h)}\|_{L^{\infty}([0,T],H^1(\R_+))} \\
\le& \lim_{h \rightarrow +\infty} \left( \|s-s_h\|_{L^{\infty}([0,T],H^1(\R_+))} + h \right)
=0.
\end{align*}
\end{proof} 
\begin{remark}\label{Boundary_function_approximation}
We have already commented that the Pearson process \eqref{eq:Pearson_SDE_boundary} is an example of process with the regularity \eqref{eq:regularity_boundary_process}. In \cite{2024AMU_numerico} a  suitable numerical scheme for the boundary process $\psi$, known as LSST scheme, which combines a smooth truncation Euler-Maruyama   method and the Lamperti transform is proposed, coupled with  forward time centred space (FTCS)  approximation of the solution for $u$ and $v$. Positiveness, boundedness, and stability are stated. See also \cite{2021_chen_Gan_wang}. 
Furthermore in \cite{2021_chen_Gan_wang} the authors proved  the strong convergence of the explicit LSST scheme up to order one and that the scheme has a unique stationary distribution which converges to the stationary distribution of the SDE model.
\end{remark}

\begin{appendices}
\section{Feynman-Kac formula}\label{appendix:proof}

In this section we derive a Feyman-Kac formula for continuous time processes  taking values in a lattice with a time dependent generator endowed with a dynamical boundary condition. Since we were not able to find a reference describing precisely this result, and since the proof of this formula is easy due to the absence of the technical difficulties of the analogous diffusion case on $\R^n$, for the convenience of the reader we provide here a sketch of the proof . 

 Let us consider a generic linear space-discrete (finite range) PDE with a dynamic boundary condition, namely
 \begin{equation}\label{eq:discretegeneric1}
 \begin{split}
& \partial_t F(t,x)  + C^+(t,x) D^+_hF(t,x)- C^-(t,x) D^-_h F(t,x) - V(t,x) F(t,x) =0\\ 
&F(T,x) = F_0(x)\\
&F(t,0) = \phi(t)
 \end{split}
 \end{equation}
for every $x \in \Omega_h^{+}$, i.e. $x > 0$, and $t\in[0,T]$. \\

\begin{remark}
For a suitable choice of the coefficients $C^{\pm},V$ equation \eqref{eq:discretegeneric1} is the most generic linear equation with nearest neighbor interaction on one-dimensional lattice and with bounded coefficients. All the results of the present section can be easily generalized to more dimensions or even on graph with more complex topological structures or even when the non-local part does not involve only the first neighbors. In the last case, the boundary condition in \eqref{eq:discretegeneric1} should fix the value of the function $F$ not only on the boundary value but on the whole set of points outside the considered domain or, better, the points outside the original domain \textit{reachable } from it considering the non local part of the equation.
\end{remark}
$\phantom{\int}$

\begin{lemma}\label{lemma:discreteeq}
Let us suppose that $F_0 \in L^{\infty}(\Omega_h^{+})$ and that $\phi, C^{\pm}, V$ are (uniformly) bounded measurable functions. Then there exists a unique solution $F \in W^{1,\infty}([0,T],L^{\infty}(\Omega_h))$ to equation \eqref{eq:discretegeneric1}.
\end{lemma}
\begin{proof}
The proof can be done by a slight modification of the proof of Lemma 7.2 in \cite{Kurtzbook}.
The existence follows from a Picard iteration argument applied in the metric space of $L^{\infty}(\Omega_h^{+})$ functions $F$ satisfying the boundary conditions in \eqref{eq:discretegeneric1}. 
The uniqueness can be proved by applying Gronwall lemma to the difference of two solutions to equation \eqref{eq:discretegeneric1}.
In both parts of the proof it is essential to use the uniform boundedness of the coefficients $\phi,C^{\pm},V$.
\end{proof}

For generic coefficients $C^{\pm}$ and $V$ equation \eqref{eq:discretegeneric1} does not generate a construction on $L^{\infty}(\Omega_h^+)$ (a condition needed to prove  Proposition \ref{proposition:sbounds}). For this reason, in the rest of this section we suppose that $C^{\pm}$ and $V$ are greater or equal than $0$. Under this hypothesis, we consider the operator $L: W^{1,\infty}([0,T], L^{\infty}(\Omega_h)) \rightarrow L^{\infty}([0,T], L^{\infty}(\Omega_h))$ defined as 
\begin{equation}\label{eq:operatorL}
L(G)(t,x)=\partial_t G(t,x) + C^+(t,x) D^+_h(G)(t,x) - C^{-}(t,x) D^-_h(G)(t,x) 
\end{equation}
for any $t \in [0,T]$ and $x > 0$, and $L(G)(t,x) = \partial_t G(t,x)$ when $x \leq 0$. \\

For $(t,x) \in [0,T] \times \Omega_h$, let $X^{t,x}_{\cdot}: [t,T] \times \Omega  \rightarrow \Omega_h $ be a family of stochastic processes (parametrized by $(t,x)\in[0,T]\times \Omega_h$) such that $X^{t,x}_t=x$.\\

\begin{definition}[Martingale problem]
  We say that $X^{t,x}_{\cdot}$ solves the martingale problem associated to the above operator $L$ if for every $F \in W^{1,\infty}([0,T],L^{\infty}(\Omega_h))$ the process $\tilde{M}_s$ on $s\in [t,T]$ defined as
\[\tilde{M}_s= F(s, X^{x,t}_{ s}) -\int_0^{ s} L(F)( \sigma, X^{x,t}_{\sigma}) d\sigma \]
is a martingale. Furthermore, We say that the martingale problem associated with $L$ has a unique (in law) solution if, for any solution $X^{t,x}_s$ of the martingale problem associated with $L$, for any $(t,x) \in [0,T] \times \Omega_h$, the law of the process $X^{t,x}_{\cdot}$ does not depend on the particular solution to the martingale problem.   
\end{definition}

\begin{lemma}\label{lemma:FK}
Under the hypothesis that $C^{\pm}$ are non-negative and bounded functions,  there exists a unique solution to the martingale problem associated with the operator $L$ defined in \eqref{eq:operatorL}.
\end{lemma}
\begin{proof}
We are going to use Theorem 7.3 of \cite{Kurtzbook}, where an analogous theorem is proved for the operator $\hat{L}$ of the form
\begin{equation}\label{eq:hatL}
\hat{L}(F)(t,x)=\partial_tF(t,x)+\int_{\Omega_h}(F(t,y)-F(t,x))\mu(t,x,d y)
\end{equation}
where $\mu$ is measurable kernel which is a positive measure with respect to the $y$ variable such that there is a measurable function $\gamma$ in $L^1([0,T])$ such that
\[\int_{\Omega_h} \mu(t,x, dy) \leq \gamma(t).\]
The operator $L$ given in equation \eqref{eq:operatorL} is of the form \eqref{eq:hatL} when the measure $\mu$ is chosen to be   
\[\mu(t,x,dy)=\mathbb{I}_{(x_0,+\infty)}(x) \left(\frac{C^{+}(t,x)}{h} \delta_{x+h}(dy)+ \frac{C^{-}(t,x)}{h} \delta_{x-h}(dy)\right), \]
which is a positive measure and satisfies the condition of Theorem 7.3 of \cite{Kurtzbook} in the particular case where $C^{\pm}$ are non-negative and uniformly bounded.
\end{proof}
We finally provide the stochastic representation formula.
\begin{lemma}[Feynman-Kac formula]\label{FKformula}
Suppose that $C^{\pm}$ and $V$ are non-negative and bounded functions and consider the stopping time $\tau^{x,t} \geq t$ which  is the hitting time of the process $X^{x,t}$ with respect to the set $\Omega_h \backslash \Omega_h^{+}$. Then the solution to equation \eqref{eq:discretegeneric1} has the following stochastic representation
\[F(t,x)= \mathbb{E}\left[e^{-\int_t^{\tau^{t,x}}V(\sigma, X^{x,t}_{\sigma}) d \sigma} \left( F_0(X^{x,t}_T) \mathbb{I}_{\tau^{x,t}> T} + \phi({\tau^{x,t}}) \mathbb{I}_{\tau^{t,x} \leq  T} \right) \right].\]
\end{lemma}
\begin{proof}
Consider the stochastic process 
\[M_{s}^{t,x}= e^{-\int_t^{\tau^{t,x}\wedge s}V(\sigma, X^{x,t}_{\sigma}) d \sigma} F(\tau^{x,t}\wedge s, X^{x,t}_{\tau^{x,t} \wedge s})\]
where $F$ is the unique solution to equation \eqref{eq:discretegeneric1} (which exists and it is unique thanks to Lemma \ref{lemma:discreteeq}). 
By Ito formula for stochastic processes with jumps (see Theorem 38.3 in \cite{Rogers_Williams_2}), we have 
\begin{align*}
d M_s^{t,x}=&- e^{-\int_t^{\tau^{t,x}\wedge s}V(\sigma, X^{x,t}_{\sigma})d \sigma  } F(\tau^{x,t}\wedge s, X^{x,t}_{\tau^{x,t} \wedge s})  V(\tau^{t,x}\wedge s, X^{x,t}_{\tau^{t,x}\wedge s}) d s\\
&+ e^{-\int_t^{\tau^{t,x}\wedge s}V(\sigma, X^{x,t}_{\sigma})d \sigma  }  L(F)(\tau^{x,t}\wedge s, X^{x,t}_{\tau^{x,t} \wedge s}) ds\\
&+ e^{-\int_t^{\tau^{t,x}\wedge s}V(\sigma, X^{x,t}_{\sigma})d \sigma  } d \tilde{M}_s\\
=&e^{-\int_t^{\tau^{t,x}\wedge s}V(\sigma, X^{x,t}_{\sigma})d \sigma  } d \tilde{M}_{s \wedge \tau^{x,t}}
\end{align*}
where 
\[\tilde{M}_s=F(s, X^{x,t}_{s}) - \int_0^s L(F)( \sigma, X^{x,t}_{ \sigma}) d \sigma,\]
and where we have extended $F$ to be $0$ when $x \leq 0$), and we have used the fact that $L(F)(t,x)-V(t,x) F(t,x)=0$ for every $x > 0$. Since $X^{t,x}_s$ solves the martingale problem associated with $L$, and $F$ is a bounded function, the process $\tilde{M}$ is a bounded martingale. Since $e^{-\int_t^{\tau^{t,x}\wedge s}V(\sigma, X^{x,t}_{\sigma})d \sigma  }$ is bounded, then $M_t$ in an $L^2$ martingale (it is actually an $L^p$ martingale for every $1\leq p< +\infty$). This means that 
\[F(t,x)=M^{t,x}_t =\mathbb{E}\left[ M_T^{t,x} \right] = \mathbb{E}\left[e^{-\int_t^{\tau^{t,x}}V(\sigma, X^{x,t}_{\sigma}) d \sigma} F(T \wedge \tau^{x,t}, X^{x,t}_{T\wedge \tau^{x,t}})\right].
\]
Using the definition of the stopping time $\tau^{x,t}$, the fact that $X_{\tau^{x,t}}^{x,t}=0$ when $\tau^{x,t}<T$ (this is due to the fact that the non-local part of equation \eqref{eq:discretegeneric1} is of nearest neighbor type) and the boundary condition in equation \eqref{eq:discretegeneric1}, we obtain the thesis.
\end{proof}

\end{appendices}

 \noindent{\bf Acknowledgements.} The research is carried out within the research project PON 2021(DM 1061, DM 1062) ``Deterministic and stochastic mathematical modelling and data analysis within the study for the indoor and outdoor impact of the climate and environmental changes for the degradation of the Cultural Heritage" of the Università degli Studi di Milano. The authors are members of GNAMPA (Gruppo Nazionale per l’Analisi Matematica, la Probabilità e le loro Applicazioni) of the Italian Istituto Nazionale di Alta Matematica (INdAM).

\bibliographystyle{plain}
\bibliography{BibliographyDegrado} 
\end{document}